\documentclass[pdflatex,sn-mathphys-num]{sn-jnl}% Math and Physical Sciences Numbered Reference Style
\usepackage{graphicx}%
\usepackage{multirow}%
\usepackage{amsmath,amssymb,amsfonts}%
\usepackage{amsthm}%
\usepackage{mathrsfs}%
\usepackage[title]{appendix}%
\usepackage{xcolor}%
\usepackage{textcomp}%
\usepackage{manyfoot}%
\usepackage{booktabs}%
\usepackage{algorithm}%
\usepackage{algorithmicx}%
\usepackage{algpseudocode}%
\usepackage{listings}%
\usepackage{enumitem}
\usepackage{pdflscape}
\usepackage{longtable}
\usepackage{float}
\usepackage{afterpage}
\usepackage{comment}

\theoremstyle{thmstyleone}%
\newtheorem{theorem}{Theorem}%  meant for continuous numbers
\theoremstyle{thmstyletwo}%
\newtheorem{remark}{Remark}%

\theoremstyle{thmstylethree}%
\newtheorem{definition}{Definition}%

\definecolor{green}{rgb}{0,0.9,0}

\definecolor{blue}{rgb}{0,0,0.9}

\def\mc{\multicolumn}

\def\trace{\mbox{Tr}}
\def\rank{\mbox{rank}}

\newcommand{\inprod}[2]{\langle #1 , #2 \rangle}

\def\norm#1{\|#1\|}

\def\inprod#1#2{\langle#1, \, #2\rangle}

\def\pobj{{\mbox{\tt pobj}}}
\def\dobj{{\mbox{\tt dobj}}}

\theoremstyle{definition}

\theoremstyle{remark}

\newtheorem{assumption}{Assumption}[section]

\numberwithin{equation}{section}

\begin{document}

\title[Convex composite optimization problems]{An augmented Lagrangian method for strongly regular minimizers in a class of convex composite optimization problems}

%%=============================================================%%
%% GivenName	-> \fnm{Joergen W.}
%% Particle	-> \spfx{van der} -> surname prefix
%% FamilyName	-> \sur{Ploeg}
%% Suffix	-> \sfx{IV}
%% \author*[1,2]{\fnm{Joergen W.} \spfx{van der} \sur{Ploeg} 
%%  \sfx{IV}}\email{iauthor@gmail.com}
%%=============================================================%%

\author[1]{\fnm{Chengjing} \sur{Wang}}\email{renascencewang@hotmail.com}

\author*[2]{\fnm{Peipei} \sur{Tang}}\email{tangpp@hzcu.edu.cn}
\equalcont{These authors contributed equally to this work.}

\affil*[1]{\orgdiv{School of Mathematics}, \orgname{Southwest Jiaotong University}, \orgaddress{\street{Xian Road}, \city{Chengdu}, \postcode{611756}, \state{Sichuan}, \country{China}}}

\affil[2]{\orgdiv{School of Computer and Computing Science}, \orgname{Hangzhou City University}, \orgaddress{\street{Huzhou Street}, \city{Hangzhou}, \postcode{310015}, \state{Zhejiang}, \country{China}}}

%%==================================%%
%% Sample for unstructured abstract %%
%%==================================%%

\abstract{In this paper, we study a class of convex composite optimization problems. We begin by characterizing the equivalence between the primal/dual strong second-order sufficient condition and the dual/primal nondegeneracy condition. Building on this foundation, we derive a specific set of equivalent conditions for the perturbation analysis of the problem. Furthermore, we employ the augmented Lagrangian method (ALM) to solve the problem and provide theoretical guarantees for its performance. Specifically, we establish the equivalence between the primal/dual second-order sufficient condition and the dual/primal strict Robinson constraint qualification, as well as the equivalence between the dual nondegeneracy condition and the nonsingularity of Clarke's generalized Jacobian for the ALM subproblem. These theoretical results form a solid foundation for designing efficient algorithms. Finally, we apply the ALM to the von Neumann entropy optimization problem and present numerical experiments to demonstrate the algorithm's effectiveness.}


%%\pacs[JEL Classification]{D8, H51}


\maketitle

\section{Introduction}\label{sec1}

In this paper, we consider the following convex composite optimization problem
\begin{align*}
(P')\qquad\min_{x\in\mathcal{R}^{n}}f(\mathcal{A}x-b)+h(\mathcal{B}x-c)
\end{align*}
where $f:\mathcal{R}^{m_{1}}\rightarrow\overline{\mathcal{R}}$ is a proper lower semicontinuous convex function which may be twice continuously differentiable, $h:\mathcal{R}^{m_{2}}\rightarrow\overline{\mathcal{R}}$ is a proper lower semicontinuous convex function,  $\mathcal{A}:\mathcal{R}^{n}\rightarrow\mathcal{R}^{m_{1}}$, $\mathcal{B}:\mathcal{R}^{n}\rightarrow\mathcal{R}^{m_{2}}$ are given linear mappings whose adjoints are denoted as $\mathcal{A}^{*}$ and $\mathcal{B}^{*}$, respectively, $b\in\mathcal{R}^{m_{1}}$ and $c\in\mathcal{R}^{m_{2}}$ are given vectors and $m=m_{1}+m_{2}$. We observe that problem $(P')$ has wide applications in various fields, such as statistics, machine learning, signal processing, image processing, and so on. One may refer to \cite{Tibshirani}, \cite{YuanL}, \cite{FriedmanHT}, \cite{BelloniCW}, \cite{GainesZ}, \cite{WangT},  \cite{LinSTW}, \cite{NakagakiFKY}, \cite{WangTHL}, and so on.

The corresponding dual problem related to problem $(P')$ is given by
\begin{eqnarray*}
	(D')\quad\quad\quad\min_{{u\in\mathcal{R}^{m_{1}}}\atop{v\in\mathcal{R}^{m_{2}}}}\Big\{f^{*}(u)+h^{*}(v)+\langle b,u\rangle +\langle c,v\rangle\, \Big|\, \mathcal{A}^{*}u+\mathcal{B}^{*}v=0\Big\}.
\end{eqnarray*}
For any $(x,u,v)\in\mathcal{R}^{n}\times\mathcal{R}^{m_{1}}\times\mathcal{R}^{m_{2}}$, the Lagrangian function corresponding to problem $(P')$ can be written as
\begin{align*}
L(x,u,v)&=\inf_{{s\in\mathcal{R}^{m_{1}}}\atop{t\in\mathcal{R}^{m_{2}}}}\Big\{f(\mathcal{A}x-b-s)+h(\mathcal{B}x-c-t)+\langle u,s\rangle+\langle v,t\rangle\Big\}\\
&=-f^{*}(u)-h^{*}(v)-\langle b,u\rangle -
\langle c,v\rangle+\langle x,\mathcal{A}^{*}u+\mathcal{B}^{*}v\rangle.
\end{align*}
The Karush-Kuhn-Tucker (KKT) condition for the composite problem $(P')$ takes the following form
\begin{eqnarray}\label{composite-problem-kkt}
\mathcal{A}^{*}u+\mathcal{B}^{*}v=0,\quad\mathcal{A}x-b\in\partial f^{*}(u),\quad \mathcal{B}x-c\in\partial h^{*}(v).
\end{eqnarray}
which can be written equivalently as
\begin{align}\label{composite-problem-kkt-generalized-equation}
\widehat{R}(x,u,v):=\left[\begin{array}{c}
\mathcal{A}^{*}u+\mathcal{B}^{*}v\\
u-\operatorname{Prox}_{f^{*}}(\mathcal{A}(x)-b+u)\\
v-\operatorname{Prox}_{h^{*}}(\mathcal{B}(x)-c+v)
\end{array}\right]=0.
\end{align}

Since Problem $(P')$ is very important, it is very meaningful to gain a deep understanding of the perturbation analysis of the problem and the algorithm for solving it.

As for perturbation analysis in optimization, many pioneers have done a lot of foundation work. In the landmark work \cite{Robinson1980}, Robinson proposed the concept of strong regularity, which is one of the most important concepts to describe the stability of the solution mapping of the following generalized equation
\begin{align}\label{eq:generalized-equation}
0\in\phi(x)+\mathcal{M}(x),
\end{align}
where  $\phi:\mathcal{R}^{t}\rightarrow\mathcal{R}^{t}$ and $\mathcal{M}:\mathcal{R}^{t}\rightrightarrows\mathcal{R}^{t}$ is a set-valued mapping. How to characterize the strong regularity of the canonically perturbed KKT system at the solution point of a given optimization problem has always been an important issue. For nonlinear programming problems, Robinson \cite{Robinson1980} proved that the strong second-order sufficient condition (SSOSC) and the linear independence constraint qualification condition (LICQ) at the solution point is a sufficient condition for the strong regularity of the canonically perturbed KKT system. Proposition 5.38 in \cite{Bonnans2000} shows that the SSOSC together with the LICQ is equivalent to the strong regularity for nonlinear programming. The characterizations of the strong regularity in terms of the strong second order sufficient condition together with the nondegeneracy condition (generalization of the LICQ condition) for second-order cone programming problems and nonlinear semidefinite programming problems are given in \cite{Bonnans2005} and \cite{Sun2006}, respectively. In addition, as the reference \cite{Clarke1976} says, the property of Clarke's generalized Jacobian of the KKT system is also an important tool to characterize the stability of the solution. In \cite{Sun2006}, the nonsingularity of Clarke's generalized Jacobian of the KKT system at the solution is also proven to be equivalent to the strong regularity. For the following general composite optimization problem
\begin{eqnarray*}
	(P)\quad\quad\quad\min_{x\in\mathcal{R}^{n}}g(F(x)),
\end{eqnarray*}
where $g:\mathcal{R}^{m}\rightarrow\overline{\mathcal{R}}$ is a lower semicontinuous proper convex function and $F:\mathcal{R}^{n}\rightarrow\mathcal{R}^{m}$ is a smooth mapping, recently under certain mild assumptions, Tang and Wang \cite{TangW2024} established the equivalence of the strong regularity of the KKT system,  the SSOSC together with the nondegeneracy condition, the nonsingularity of Clarke's generalized Jacobian of the nonsmooth system at the KKT point.

Note that problem $(P')$ can be regarded as a special case of problem $(P)$ with $F(x):=(\mathcal{A}x-b,\mathcal{B}x-c)$ and $g(a,b):=f(a)+h(b)$. Hence it is very interesting to study whether there is a specific equivalence characterization of the perturbation analysis of problem $(P')$. Let $(\bar{x},\bar{\mu}) $ be a solution of the KKT system \eqref{composite-problem-kkt-generalized-equation} with $\bar{\mu}:=(\bar{u},\bar{v})$, $\bar{u}\in\mathcal{R}^{m_{1}}$ and $\bar{v}\in\mathcal{R}^{m_{2}}$.

Although there are numerous algorithms for solving problem $(P')$,  the augmented Lagrangian method (ALM) which dates back to \cite{Hestenes1969,Powell1969} and is further studied in \cite{Rockafellar1976-1,Rockafellar1976-2,Luque1984} is an efficient approach, especially for large-scale problems. An appealing feature lies in the asymptotic superlinear convergence rate of the dual sequence generated by the ALM under the metric subregularity at the dual solution \cite{CuiSun2019}. Meanwhile, the metric subregularity is equivalent to the calmness property (see 3H of \cite{DontchevandRockafellar2009}). It is also known from \cite{Francisco2008} that the metric subregularity of the dual solution mapping at a dual optimal solution for the origin is equivalent to the quadratic growth condition at the corresponding dual optimal solution. We need the quadratic growth condition to guarantee the superlinear convergence rate. It is usually difficult to check the quadratic growth condition directly and one may deal with it alternatively by looking for other sufficient conditions. One of such conditions is the second order sufficient condition (SOSC). As shown in Theorem 3.109 of \cite{Bonnans2000}, under the conditions that $g$ is a lower semicontinuous proper convex function and the set $\operatorname{epi}g$ outer second order regular, the SOSC implies the quadratic growth condition. Other generic properties including the strict Robinson constraint qualification (SRCQ), the constraint nondegeneracy condition are also important for the corresponding problem. A natural question is whether the other conditions have any relationships with the SOSC. Furthermore, during the iterations of the ALM, we usually need to find an approximate solution of the corresponding inner subproblem efficiently with a given accuracy. If the Hessian of the inner subproblem is invertible we may apply some existing second order method such as the semismooth Newton (SSN) method to find the desired solution efficiently with a high accuracy. How to guarantee and characterize the nonsingularity of the Hessian is another important problem.

In this paper, we make an assumption that some functions are $C^{2}$-cone reducible. It needs to emphasize that the set of the $C^{2}$-cone reducible functions is rich and includes all the indicator functions of $C^{2}$-cone reducible sets. Under the assumption of $C^{2}$-cone reducibility and other implementable assumptions, we will prove the equivalence between the primal/dual SSOSC and the dual/primal nondegeneracy condition, based on which we will establish a specific set of equivalent conditions to characterize the perturbation analysis of problem $(P')$. As for the ALM applying to solve problem $(P')$, we will establish the equivalence between the primal/dual SOSC and the dual/primal SRCQ. Moreover, we will prove the equivalence between the dual nondegeneracy condition and the nonsingularity of Clarke's generalized Jacobian of the subproblem of the ALM. All the results are definitely important in both the computational and theoretical study of general composite optimization problems.

The remaining parts of this paper are organized as follows. In Section \ref{sec:preliminaries}, we summarize some preliminaries from variational analysis. In Section \ref{sec:Perturbation analysis}, we establish a specific set of equivalent conditions for problem ($P'$). In Section \ref{sec:SSN-ALM}, we introduce the ALM for problem ($P'$) and conduct a detailed theoretical analysis. In Section \ref{sec:entropy}, we focus on applying the ALM to the von Neumann entropy optimization problem. In Section \ref{sec:Numerical experiments}, we present the numerical experiments. In Section \ref{sec:Conclusion}, we give our conclusion.

\subsection{Additional notations}
\label{subsec:notations}
In our paper, the closed unit ball in $\mathcal{R}^{m}$ is denoted by $\mathbb{B}_{\mathcal{R}^{m}}$, while for a given point $x\in\mathcal{R}^{m}$ and $\varepsilon>0$, $\mathbb{B}(x,\varepsilon):=\{u\in\mathcal{R}^{m}\, |\, \|u-x\|\leq\varepsilon\}$. Given a set $C\subseteq\mathcal{R}^{m}$,  the indicator function $\delta_{C}$ of the set $C$ is defined by $\delta_{C}(x)=0$ if $x\in C$, otherwise $\delta_{C}(x)=+\infty$. Given a point $x\in\mathcal{R}^{m}$, $\operatorname{dist}(x,C)$ denotes the distance from $x$ to the set $C$, $\operatorname{\Pi}_{C}(x)$ denotes the projection of $x$ onto $C$. Denote $\mathbb{N}:=\{1,2,\ldots\}$ and $\mathbb{N}_{\infty}:=\{N\subseteq\mathbb{N}\, |\, \mathbb{N}\setminus N\ \mbox{is finite}\}$. For a sequence $\{x^{k}\}$,  $x^{k}{\stackrel{C}{\rightarrow}}x$ means that $x^{k}\rightarrow x$ with $x^{k}\in C$ and given $N\subseteq\mathbb{N}$, $x^{k}{\stackrel{N}{\rightarrow}}x$ stands for $x^{k}\rightarrow x$ with $k\in N$. The smallest cone containing $C$, which is called the positive hull of $C$, has the formula $\operatorname{pos}C=\{0\}\cup\{\lambda x\, |\, x\in C,\, \lambda>0\}$. If $C=\emptyset$, one has $\operatorname{pos}C=\{0\}$ and if $C\neq\emptyset$, one has $\operatorname{pos}C=\{\lambda x\, |\, x\in C,\, \lambda\geq0\}$. The intersection of all the convex sets containing $C$ is called the convex hull of $C$ and is denoted by $\operatorname{conv}C$. For a nonempty closed convex cone $C$, the polar cone of $C$ is $C^{\circ}=\{y\,|\, \langle x,y\rangle\leq 0,\,\forall\, x\in C\}$,  the lineality space of $C$ is $\operatorname{lin}C:=C\cap(-C)$, which is the largest subspace contained in $C$, and the affine space is $\operatorname{aff}C:=C-C$, which is the smallest subspace containing $C$. Given a matrix $A\in\mathcal{R}^{m\times n}$, $\operatorname{NULL}(A):=\{x\in\mathcal{R}^{n}\,|\,Ax=0\}$.

\section{Preliminaries}\label{sec:preliminaries}

In this section, we introduce some basic knowledge on variational analysis that are used in the subsequent sections. More details can be refereed in the monographs \cite{Bonnans2000,Rockafellar1998,Mordukhovich2006,Mordukhovich2018}.

Let $r:\mathcal{R}^{m}\rightarrow\overline{\mathcal{R}}$ be a real valued function with its epigraph $\operatorname{epi}r$ defined as the set
\begin{eqnarray*}
	\operatorname{epi}r:=\Big\{(x,c)\, \Big|\, x\in\operatorname{dom}r,\ c\in\mathcal{R},\ r(x)\leq c\Big\}
\end{eqnarray*}
and its conjugate at $x\in\mathcal{R}^{m}$ defined by
\begin{eqnarray*}
	r^{*}(x):=\sup_{u\in\operatorname{dom}r}\Big\{\langle x,u\rangle-r(u)\Big\},
\end{eqnarray*}
where $\operatorname{dom}r:=\{x\,|\,r(x)<+\infty\}$.
The extended function $r$ is said to be proper if $r(x)>-\infty$ and $r(x)\not\equiv +\infty$ for all $x\in\mathcal{R}^{m}$. For a proper lower semicontinuous function $r$, the Moreau envelope function $e_{\sigma r}$ and the proximal mapping $\operatorname{Prox}_{\sigma r}$ corresponding to $r$ with a parameter $\sigma>0$ are defined as
\begin{eqnarray*}
	e_{\sigma r}(x)&:=&\inf_{u\in\mathcal{R}^{m}}\Big\{r(u)+\frac{1}{2\sigma}\|u-x\|^{2}\Big\},\\
	\operatorname{Prox}_{\sigma r}(x)&:=&\mathop{\operatorname{argmin}}_{u\in\mathcal{R}^{m}}\Big\{r(u)+\frac{1}{2\sigma}\|u-x\|^{2}\Big\}.
\end{eqnarray*}
%Furthermore, if $r$ is a lower semicontinuous proper convex function, it follows from Theorem 2.26 of \cite{Rockafellar1998} that the Moreau envelope function $e_{\sigma r}$ is also convex and continuously differentiable with
%\begin{eqnarray*}
%	\nabla e_{\sigma r}(x)=\frac{1}{\sigma}\big(x-\operatorname{Prox}_{\sigma r}(x)\big).
%\end{eqnarray*}
%The proximal mapping $\operatorname{Prox}_{\sigma r}$ is single-valued, Lipschitz continuous, and satisfies the following Moreau identity (see e.g., Theorem 31.5 of \cite{Rockafellar1970})
%\begin{eqnarray*}
%	\operatorname{Prox}_{\sigma r}(x)+\sigma\operatorname{Prox}_{\sigma^{-1}r^{*}}(\sigma^{-1}x)=x,\quad \forall\, x\in\mathcal{R}^{m}.
%\end{eqnarray*}

For the extended real-valued function $r$ and $x\in\operatorname{dom}r$, the regular/Fr\'{e}chet subdifferential of $r$ at $x$ is the set
\begin{align*}
\widehat{\partial}r(x):=\Big\{v\in\mathcal{R}^{m}\,\Big|\, r(x')\geq r(x)+\langle v,x'-x\rangle+o(\|x'-x\|)\Big\}.
\end{align*}
The limiting/Mordukhovich subdifferential of $r$ at $x$, denoted by $\partial r(x)$, is given by
\begin{align*}
\partial r(x):=\Big\{v\in\mathcal{R}^{m}\,\Big|\, \exists\, x^{k}\rightarrow x\ \mbox{with}\ r(x^{k})\rightarrow r(x),\, v^{k}\in\widehat{\partial}r(x^{k})\ \mbox{with}\ v^{k}\rightarrow v\Big\}.
\end{align*}
%The singular subdifferential of $r$ at $x$ is defined by
%\begin{align*}
%	\partial^{\infty}r(x):=\Big\{v\in\mathcal{R}^{m}\,\Big|\,&\exists\, t_{k}\downarrow 0,\ x^{k}\rightarrow x\ \mbox{with}\ r(x^{k})\rightarrow r(x),\\
%	&v^{k}\in\widehat{\partial}r(x^{k})\ \mbox{and}\ t_{k}v^{k}\rightarrow v\Big\}.
%\end{align*}

Let $x\in\mathcal{R}^{m}$ be a point such that $r(x)$ is finite. The lower and upper directional epiderivatives of $r$ are defined as follows
\begin{eqnarray*}
	r^{\downarrow}_{-}(x,d)&:=&\liminf_{{t\downarrow 0}\atop{d'\rightarrow d}}\frac{r(x+td')-r(x)}{t},\\
	r^{\downarrow}_{+}(x,d)&:=&\sup_{\{t_{n}\}\in\operatorname{\Sigma}}\left(\liminf_{{n\rightarrow\infty}\atop{d'\rightarrow d}}\frac{r(x+t_{n}d')-r(x)}{t_{n}}\right),
\end{eqnarray*}
where $\operatorname{\Sigma}$ denotes the set of positive real sequences $\{t_{n}\}$ converging to zero. We say that $r$ is directionally epidifferentiable at $x$ in a direction $d$ if $r^{\downarrow}_{-}(x,d)=r^{\downarrow}_{+}(x,d)$. If $r$ is convex, then $r^{\downarrow}_{-}(x,d)=r^{\downarrow}_{+}(x,d)$. Assuming that $r(x)$ and the respective directional epiderivatives $r^{\downarrow}_{-}(x,d)$ and $r^{\downarrow}_{+}(x,d)$ are finite, we can define the lower and upper second order epiderivatives as
\begin{eqnarray*}
	r^{\downdownarrows}_{-}(x;d,w)&:=&\liminf_{{t\downarrow 0}\atop{w'\rightarrow w}}\frac{r(x+td+\frac{1}{2}t^{2}w')-r(x)-tr^{\downarrow}_{-}(x,d)}{\frac{1}{2}t^{2}},\\
	r^{\downdownarrows}_{+}(x;d,w)&:=&\sup_{\{t_{n}\}\in\operatorname{\Sigma}}\left(\liminf_{{n\rightarrow\infty}\atop{w'\rightarrow w}}\frac{r(x+t_{n}d+\frac{1}{2}t_{n}^{2}w')-r(x)-t_{n}r^{\downarrow}_{+}(x,d)}{\frac{1}{2}t_{n}^{2}}\right).
\end{eqnarray*}
The function $r$ is parabolically epidifferentiable at $x$ in a direction $d$ if $r^{\downdownarrows}_{-}(x;d,\cdot)=r^{\downdownarrows}_{+}(x;d,\cdot)$.

Given a set $C\subseteq\mathcal{R}^{m}$ and a point $x\in C$, recall that the (Painlev\'{e}-Kuratowski) outer and inner limit of the set-valued mapping $S:\mathcal{R}^{m}\rightrightarrows\mathcal{R}^{n}$ as $x'\rightarrow x$ are defined by
\begin{align*}
\limsup_{x'\rightarrow x} S(x):=\Big\{y\in\mathcal{R}^{n}\, \Big|\, \exists\, x^{k}\rightarrow x,\ y^{k}\rightarrow y\ \mbox{with}\ y^{k}\in S(x^{k})\Big\}
\end{align*}
and
\begin{align*}
\liminf_{x'\rightarrow x} S(x):=\left\{y\in\mathcal{R}^{n}\, \left|\, \forall\, x^{k}\rightarrow x,\ \exists\, N\in\mathbb{N}_{\infty},\  y^{k}\stackrel{N}{\longrightarrow} y\ \mbox{with}\ y^{k}\in S(x^{k})\right.\right\},
\end{align*}
respectively.
The (Bouligand) tangent/contingent cone to $C$ at $x$ is defined by
\begin{align*}
\mathcal{T}_{C}(x):=\limsup_{t\downarrow 0}\frac{C-x}{t}=\Big\{y\in\mathcal{R}^{m}\, \Big|\, \exists\ t_{k}\downarrow 0,\ y^{k}\rightarrow y\ \mbox{with}\ x+t_{k}y^{k}\in C\Big\}.
\end{align*}
The corresponding regular tangent cone to $C$ at $x$ is defined by
\begin{align*}
\widehat{\mathcal{T}}_{C}(x):=\liminf_{{t\downarrow 0, x'{\stackrel{C}{\rightarrow}}x}}\frac{C-x'}{t}.
\end{align*}
The regular/Fr\'{e}chet normal cone to $C$ at $x$ can be defined equivalently by
\begin{align*}
\widehat{\mathcal{N}}_{C}(x):=\left\{y\in\mathcal{R}^{m}\, \left|\, \limsup_{x'{\stackrel{C}{\rightarrow}} x}\frac{\langle y,x'-x\rangle}{\|x'-x\|}\leq 0\right.\right\}=\mathcal{T}_{C}^{\circ}(x).
\end{align*}
The limiting/Mordukhovich normal cone to $C$ at $x$ admits the following form
\begin{align*}
\mathcal{N}_{C}(x):=\limsup_{x'{\stackrel{C}{\rightarrow}} x}\widehat{\mathcal{N}}_{C}(x'),
\end{align*}
which is equivalent to the original definition by Mordukhovich \cite{MORDUKHOVICH1976}, i.e.,
\begin{align*}
\mathcal{N}_{C}(x):=\limsup_{x'\rightarrow x}\left\{\operatorname{pos}(x'-\operatorname{\Pi}_{C}(x'))\right\},\quad \mathcal{N}_{C}(x)^{\circ}=\widehat{\mathcal{T}}_{C}(x),
\end{align*}
if $C$ is locally closed at $x$.

When the set $C$ is convex, both tangent and regular tangent cones reduce to the classical tangent cone while both regular and limiting normal cones reduce to the classical normal cone of convex analysis with $\mathcal{T}_{C}(x)$ and $\mathcal{N}_{C}(x)$ as the commonly used notations, respectively.

\begin{definition}
	Let $\mathcal{X}$ be a finite dimensional Euclidean space. A set $C\subseteq\mathcal{R}^{m}$ is said to be outer second order regular at a point $x\in C$ in a direction $d\in\mathcal{T}_{C}(x)$ and with respect to a linear mapping $\mathcal{W}:\mathcal{X}\rightarrow\mathcal{R}^{m}$ if for any sequence $x^{n}\in C$ of the form $x^{n}:=x+t_{n}d+\frac{1}{2}t_{n}^{2}r^{n}$, where $t_{n}\downarrow 0$ and $r^{n}=\mathcal{W}w^{n}+a^{n}$ with $\{a^{n}\}$ being a convergent sequence in $\mathcal{R}^{m}$ and $w^{n}$ being a sequence in $\mathcal{X}$ satisfying $t_{n}w^{n}\rightarrow 0$, the following condition holds:
	\begin{align}\label{def-outer-second-order-regular}
	\lim_{n\rightarrow\infty}\operatorname{dist}\left(r^{n},\mathcal{T}^{2}_{C}(x,d)\right)=0,
	\end{align}
	where $\mathcal{T}^{2}_{C}(x,d)$ is the outer second order tangent set to the set $C$ at the point $x$ and in the direction $d$ and it is defined by
	\begin{align*}
	\mathcal{T}_{C}^{2}(x,d):=\limsup_{t\downarrow 0}\frac{C-x-td}{\frac{1}{2}t^{2}}.
	\end{align*}
	If $C$ is outer second order regular at $x\in C$ in every direction $d\in\mathcal{T}_{C}(x)$ and with respect to any $\mathcal{X}$ and $\mathcal{W}$, i.e., \eqref{def-outer-second-order-regular} holds for any sequence $x+t_{n}d+\frac{1}{2}t_{n}^{2}r^{n}\in C$ such that $t_{n}r^{n}\rightarrow 0$ and $d\in\mathcal{T}_{C}(x)$, we say that $C$ is outer second order regular at $x$.
\end{definition}

\begin{definition}
	Let $r:\mathcal{R}^{m}\rightarrow\overline{\mathcal{R}}$ be an extended real valued function taking a finite value at a point $x$. We say that $r$ is outer second order regular at the point $x$ and in the direction $d$, if $r^{\downarrow}_{-}(x,d)$ is finite and the set $\operatorname{epi}r$ is outer second order regular at the point $(x,r(x))$ in the direction $(d,r^{\downarrow}_{-}(x,d))$. We say that $r$ is outer second order regular at the point $x$ if the set $\operatorname{epi}r$ is outer second order regular at the point $(x,r(x))$.
\end{definition}

%For $x\in\operatorname{dom}r$ and $d$ satisfying $r^{\downarrow}_{-}(x;d)=\langle u,d\rangle$ with $d^{2}r(x,u)(d)>-\infty$, the function $r$ is parabolically regular at $x$ in the direction $d$ for $u$ if the function $r$ is outer second order regular at $x$ in the direction $d$ due to Proposition 3.103 of \cite{Bonnans2000}.

\begin{definition}
	Let $C\subseteq\mathcal{R}^{m}$ and $K\subseteq\mathcal{R}^{t}$ be convex closed sets. We say that the set $C$ is $C^{\ell}$-reducible to the set $K$, at a point $x\in C$, if there exist a neighborhood $N$ at $x$ and an $\ell$-times continuously differentiable mapping $\Xi:N\rightarrow\mathcal{R}^{t}$ such that
	\begin{enumerate}
		\item $\Xi'(x):\mathcal{R}^{m}\rightarrow\mathcal{R}^{t}$ is onto,
		\item $C\cap N=\{x\in N\, |\, \Xi(x)\in K\}$.
	\end{enumerate}
	We say that the reduction is pointed if the tangent cone $\mathcal{T}_{K}(\Xi(x))$ is a pointed cone. If, in addition, the set $K-\Xi(x)$ is a pointed closed convex cone, we say that $C$ is $C^{\ell}$-cone reducible at $x$. We can assume without loss of generality that $\Xi(x)=0$.
\end{definition}	

We say a closed proper convex function $r:\mathcal{R}^{m}\rightarrow\overline{\mathcal{R}}$ is $C^{2}$-cone reducible at $x$, if the set $\operatorname{epi}r$ is $C^{2}$-cone reducible at $(x,r(x))$. Moreover, we say $r$ is $C^{2}$-cone reducible if $r$ is $C^{2}$-cone reducible at every point $x\in\operatorname{dom}r$. It is known from Proposition 3.136 of \cite{Bonnans2000} that $r$ is outer second order regular at $x$ if $r$ is $C^{2}$-cone reducible at $x$.
%
%Recall that a function $r:\mathcal{R}^{m}\rightarrow\overline{\mathcal{R}}$ is locally Lipschitz continuous around $\bar{x}\in\operatorname{dom}r$ relative to some set $\Omega\subseteq\operatorname{dom}r$ if there exists a constant $\kappa\in(0,+\infty)$ and a neighborhood $\mathcal{U}$ of $\bar{x}$ such that
%	\begin{align*}
%		|r(x)-r(y)|\leq\kappa\|x-y\|,\, \forall\, x,y\in\Omega\cap\mathcal{U}.
%\end{align*}

For a set valued mapping $S:\mathcal{R}^{m}\rightrightarrows\mathcal{R}^{n}$, the range of $S$ is taken to be the set $\operatorname{rge}S:=\{u\in\mathcal{R}^{n}\,|\,\exists\,x\in\mathcal{R}^{m}\ \mbox{with}\ u\in S(x)\}$. The inverse mapping $S^{-1}:\mathcal{R}^{n}\rightrightarrows\mathcal{R}^{m}$ is defined by $S^{-1}(u):=\{x\in\mathcal{R}^{m}\,|\,u\in S(x)\}$.
%We define  the corresponding graphical derivative and coderivative constructions generated
%by the tangent cone, the regular normal cone and the limiting normal cone, respectively. Given a point $x\in\operatorname{dom}S$, the graphical derivative of $S$ at $x$ for any $u\in S(x)$ is the mapping $DS(x,u):\mathcal{R}^{m}\rightrightarrows\mathcal{R}^{n}$ defined by
%\begin{eqnarray*}
%	z\in DS(x,u)(w)\quad\Longleftrightarrow\quad (w,z)\in\mathcal{T}_{\operatorname{gph}S}(x,u),
%\end{eqnarray*}
%where $\operatorname{gph}S$ is the graph of $S$ a subset of $\mathcal{R}^{m}\times\mathcal{R}^{n}$, namely $\operatorname{gph}S:=\{(x,y)\, |\, y\in S(x)\}$.
The limiting coderivative is the mapping $D^{*}S(x,u):\mathcal{R}^{n}\rightrightarrows\mathcal{R}^{m}$ defined by
\begin{align*}
v\in D^{*}S(x,u)(y)\quad\Longleftrightarrow\quad (v,-y)\in\mathcal{N}_{\operatorname{gph}S}(x,u).
\end{align*}
%The regular coderivative $\widehat{D}^{*}S(x,u):\mathcal{R}^{n}\rightrightarrows\mathcal{R}^{m}$ is defined by
%\begin{align*}
%	v\in \widehat{D}^{*}S(x,u)(y)\quad\Longleftrightarrow\quad (v,-y)\in\widehat{\mathcal{N}}_{\operatorname{gph}S}(x,u).
%\end{align*}
Here the notations $D^{*}S(x,u)$ is simplified to $D^{*}S(x)$, when $S$ is single-valued at $x$, i.e., $S(x)=\{u\}$.

\begin{definition}
	The multifunction $S:\mathcal{R}^{m}\rightrightarrows\mathcal{R}^{n}$ is said to be calm at $\bar{z}$ for $\bar{w}$ if there exists $\kappa\geq0$ along
	with $\varepsilon>0$ and $\delta>0$ such that for all $z\in\mathbb{B}(\bar{z},\varepsilon)$,
	\begin{align*}
	S(z)\cap\mathbb{B}(\bar{w},\delta)\subseteq S(\bar{z})+\kappa\|z-\bar{z}\|\mathbb{B}_{\mathcal{R}^{n}}.
	\end{align*}
\end{definition}	
\begin{definition}
	The multifunction $S:\mathcal{R}^{m}\rightrightarrows\mathcal{R}^{n}$ is said to be metrically subregular at $\bar{z}$ for $\bar{w}$ if there
	exists $\kappa\geq0$ along with $\varepsilon>0$ and $\delta>0$ such that for all  $z\in\mathbb{B}(\bar{z},\varepsilon)$,
	\begin{align*}
	\operatorname{dist}(z,S^{-1}(\bar{w}))\leq\kappa\operatorname{dist}(\bar{w},S(z)\cap\mathbb{B}(\bar{w},\delta)).
	\end{align*}
\end{definition}

It is known from Theorem 3H.3 of \cite{DontchevandRockafellar2009} that $S$ is calm at $\bar{z}$ for $\bar{w}$ if and only if that $S^{-1}$ is metrically subregular at $\bar{w}$ for $\bar{z}$.

For a vector-valued locally Lipschitz continuous function $F:\mathcal{O}\rightarrow\mathcal{R}^{n}$ with $\mathcal{O}$ an open subset of $\mathcal{R}^{m}$, we know from Rademacher's theorem (see Theorem 9.60 of \cite{Rockafellar1998}) that $F$ is F(r\'{e}chet)-differentiable almost everywhere on $\mathcal{O}$. Let $\mathcal{D}_{F}$ be the set of all points where $F$ is F-differentiable and $F'(x)\in\mathcal{R}^{n\times m}$ be the Jacobian of $F$ at $x\in \mathcal{D}_{F}$, whose adjoint is denoted as $F'(\bar{x})^{*}$. Then the B-subdifferential of $F$ at $x\in\mathcal{O}$ is defined by
\begin{align*}
\partial_{B}F(x):=\Big\{U\in\mathcal{R}^{n\times m}\,\Big|\, \exists\, x^{k}\stackrel{\mathcal{D}_{F}}{\longrightarrow} x, F'(x^{k})\rightarrow U\Big\}.
\end{align*}
The Clarke subdifferential of $F$ at $x$ is defined as the convex hull of the B-subdifferential of $F$ at $x$, that is, $\partial F(x):=\operatorname{conv}(\partial_{B}F(x))$.

\begin{definition}
	We say a function $F:\mathcal{R}^{m}\rightarrow\mathcal{R}^{n}$ is directionally differentiable at $x\in\mathcal{R}^{m}$ in a direction $d\in\mathcal{R}^{m}$, if
	\begin{align*}
	F'(x,d):=\lim_{t\downarrow 0}\frac{F(x+td)-F(x)}{t}
	\end{align*}
	exists. If $F$ is directionally differentiable at $x$ in every direction $d\in\mathcal{R}^{m}$, we say $F$ is directionally differentiable at $x$.
\end{definition}	

%Due to properties of the graphical derivative, Proposition 2.2 of \cite{Hoheisel2012}, we have $DF(x)(d)=\{F'(x,d)\}$ for all $d\in\mathcal{R}^{m}$, if $F$ is directionally differentiable at $x$.

\begin{definition}
	A function $F:\mathcal{R}^{m}\rightarrow\mathcal{R}^{n}$ is said to be semismooth at $x\in\mathcal{R}^{m}$ if $F$ is locally Lipschitzian at $x$ and
	\begin{align*}
	\lim_{{U\in\partial F(x+th')}\atop{h'\rightarrow h,t\downarrow 0}}\{Uh'\}
	\end{align*}
	exists for all $h\in\mathcal{R}^{m}$. If for any $V\in\partial F(x+h)$ and $h\rightarrow 0$,  $Vh-F'(x,h)=O(\|h\|^{1+p})$, where $0<p\leq 1$, then we call $F$ is $p$-order semismooth at $x$. If $p=1$, we also call $F$ is strongly semismooth at $x$.
\end{definition}	

It is known from  \cite{QiandSun1993} that $F$ is directionally differentiable at $x\in\mathcal{R}^{m}$ if $F$ is semismooth at $x$.

For problem ($P$), the SOSC at $\bar{x}$ can be written as
\begin{align*}
&\sup_{\mu\in\Lambda(\bar{x})}\Big\{\langle\mu,F''(\bar{x})(d,d)\rangle+\varphi_{g}(F(\bar{x}),\mu)(F'(\bar{x})d)\Big\}> 0,\, \forall\, d\in\mathcal{C}(\bar{x})\setminus\{0\},
\end{align*}
where
\begin{align*}
\varphi_{g}(
F(\bar{x}),\mu)(F'(\bar{x})d):=\inf\limits_{w\in\mathcal{R}^{m}}\Big\{g^{\downdownarrows}_{-}(F(\bar{x});F'(\bar{x})d,w)-\langle w,\mu\rangle\Big\},
\end{align*}
\begin{align*}
\mathcal{C}(\bar{x})&=\Big\{d\, \Big|\, g^{\downarrow}_{-}(F(\bar{x}),F'(\bar{x})d)=0\Big\}=\Big\{d\, \Big|\, F'(\bar{x})d\in\mathcal{N}_{\partial g(F(\bar{x}))}(\mu)\Big\},
\end{align*}
and $\Lambda(\bar{x})$ is the set of all Lagrangian multipliers associated with $\bar{x}$.

Let $r:\mathcal{R}^{m}\rightarrow\overline{\mathcal{R}}$. Given $\bar{x}\in\operatorname{dom}r$ and $\bar{u}\in\partial r(\bar{x})$, we define the function  $\Gamma_{r}(\bar{x},\bar{u}):\mathcal{R}^{m}\rightarrow\overline{\mathcal{R}}$ as follows
\begin{align*}
\Gamma_{r}(\bar{x},\bar{u})(v):=\left\{\begin{array}{cl}\min\limits_{{d\in\mathcal{R}^{m},v=Ud}\atop{U\in\partial \operatorname{Prox}_{r}(\bar{x}+\bar{u})}}\langle v,d-v\rangle,&\mbox{if}\ v\in\bigcup\limits_{U\in\partial \operatorname{Prox}_{r}(\bar{x}+\bar{u})}\operatorname{rge}U,\\
+\infty,&\mbox{otherwise}.
\end{array}\right.
\end{align*}
Based on \cite{TangW2024}, we present the definition of the SSOSC as below.
\begin{definition}
	Let $\bar{x}$ be a stationary point of problem $(P)$. We say that the SSOSC holds at $\bar{x}$ if
	\begin{align}
	\sup_{\mu\in\Lambda(\bar{x})}\Big\{\langle\mu,F''(\bar{x})(d,d)\rangle+\Gamma_{g}(F(\bar{x}),\mu)(F'(\bar{x})d)\Big\}> 0,\,\forall\,d\neq0.\label{eq:comp-prob-SSOSC}
	\end{align}
\end{definition}

We present the following assumptions which will be required for some theoretical results.
\begin{assumption}\label{assump-1}
	Suppose that $r:\mathcal{R}^{m}\rightarrow\overline{\mathcal{R}}$ is a proper convex $C^{2}$-cone reducible function.
\end{assumption}	
\begin{assumption}\label{assumption-2}
	Let $r:\mathcal{R}^{m}\rightarrow\overline{\mathcal{R}}$ be a lower semicontinuous proper convex function, $\bar{x}\in\operatorname{dom}r$ and $\bar{u}\in\partial r(\bar{x})$. Suppose that
	\begin{align*}
	\operatorname{dom}D^{*}(\partial r)(\bar{x},\bar{u})&=\operatorname{aff}\Big\{d\, \Big|\, r^{\downarrow}_{-}(\bar{x},d)=\langle\bar{u},d\rangle\Big\},\\
	\Big\{d\, \Big|\, 0\in D^{*}\operatorname{Prox}_{r}(\bar{x}+\bar{u})(d)\Big\}&=\operatorname{aff}\Big\{d\, \Big|\, \operatorname{Prox}_{r}'(\bar{x}+\bar{u},d)=0\Big\}.
	\end{align*}
\end{assumption}	
\begin{assumption}\label{assumption-3}
	Let $r:\mathcal{R}^{m}\rightarrow\overline{\mathcal{R}}$ be a lower semicontinuous proper convex function, $\bar{x}\in\operatorname{dom}r$ and $\bar{u}\in\partial r(\bar{x})$. For any $v\in\operatorname{rge}(D^{*}\operatorname{Prox}_r(\bar{x}+\bar{u}))$, suppose that
	\begin{align*}
	\mathop{\operatorname{argmin}}_{{d\in\mathcal{R}^{m},v=Ud}\atop{U\in\partial \operatorname{Prox}_{r}(\bar{x}+\bar{u})}}\langle v,d-v\rangle\cap\Big\{d\,\Big|\, v\in D^{*}\operatorname{Prox}_{r}(\bar{x}+\bar{u})(d)\Big\}\neq\emptyset.
	\end{align*}
\end{assumption}

\begin{remark}
	We say a function $r$ satisfies Assumption \ref{assump-1}, Assumption \ref{assumption-2} or \ref{assumption-3} at $\bar{x}$ for $\bar{u}$. We also say a function $r$ satisfies Assumption \ref{assumption-2} or \ref{assumption-3} for short if the corresponding points are known from the context. In Appendix A of \cite{TangW2024}, based on Assumption \ref{assump-1} we have given sufficient conditions for Assumptions \ref{assumption-2} and \ref{assumption-3}, respectively. In fact, the above assumptions hold for many functions in conic programming, including the $\ell_{p}(p=1,2,\infty)$ norm function, the indicator functions of the nonnegative orthant cone, the second-order cone, and the positive semidefinite cone.  We have also taken the indicator function of the positive semidefinite cone, which is the most difficult case, as an example to provide a verification of Assumptions \ref{assumption-2} and \ref{assumption-3} in Appendix B of \cite{TangW2024}.
\end{remark}

\section{The equivalence of the nondegeneracy condition and the SSOSC}
\label{sec:Perturbation analysis}

In this section, we are dedicated to study the perturbation analysis of problem $(P')$.
We focus on proving the equivalence of the nondegeneracy condition of the primal/dual problem and the SSOSC of the dual/primal problem for the problems $(P')$ and $(D')$.
The nondegeneracy condition and the SSOSC are two kinds of conditions which are very important to characterize the nonsingularity of the Hessian for the inner subproblem of the ALM (see Section \ref{sec:SSN-ALM} for more details).
Let $\bar{x}$ be an optimal solution of the primal problem $(P')$, $\Lambda(\bar{x})$ be the set of all such Lagrangian multipliers $(\bar{u},\bar{v})$ associated with $\bar{x}$ and $(\bar{u},\bar{v})$ be a solution of the dual problem $(D')$, $\Lambda(\bar{u},\bar{v})$ be the set of all such Lagrangian multipliers $\bar{x}$ associated with $(\bar{u},\bar{v})$. The SSOSC of the primal/dual problem may look very complicated, we will give a very simple characterization when the set $\Lambda(\bar{x})/\Lambda(\bar{u},\bar{v})$ is a singleton.

The nondegeneracy conditions for the primal problem $(P')$ at $\bar{x}$ and the dual problem $(D')$ at $(\bar{u},\bar{v})$ are
\begin{align}
&\left(\begin{array}{cc}\mathcal{A}\\\mathcal{B}\end{array}\right)\mathcal{R}^{n}-\left(\begin{array}{cc}\Big\{d\, \Big|\, f^{\downarrow}_{-}(\mathcal{A}\bar{x}-b,\mathcal{A}d)=-f^{\downarrow}_{-}(\mathcal{A}\bar{x}-b,-\mathcal{A}d)\Big\}\\\Big\{d\, \Big|\, h^{\downarrow}_{-}(\mathcal{B}\bar{x}-c,\mathcal{B}d)=-h^{\downarrow}_{-}(\mathcal{B}\bar{x}-c,-\mathcal{B}d)\Big\}\end{array}\right)=\left(\begin{array}{cc}\mathcal{R}^{m_{1}}\\\mathcal{R}^{m_{2}}\end{array}\right)\label{eq:nondeg-cond-primal}
\end{align}
and
\begin{align}	
&\mathcal{A}^{*}\Big\{d_{1}\, \Big|\, f^{*\downarrow}_{-}(\bar{u},d_{1})=-f^{*\downarrow}_{-}(\bar{u},-d_{1})\Big\}+\mathcal{B}^{*}\Big\{d_{2}\, \Big|\, h^{*\downarrow}_{-}(\bar{v},d_{2})=-h^{*\downarrow}_{-}(\bar{v},-d_{2})\Big\}=\mathcal{R}^{n},\label{eq:nondeg-cond-dual}
\end{align}
respectively. We can write out the SSOSC for the primal problem $(P')$ as
\begin{align}\label{eq:SSOSC-primal}
&\sup_{(u,v)\in\Lambda(\bar{x})}\Big\{\Gamma_{f}(\mathcal{A}\bar{x}-b,u)(\mathcal{A}d)+\Gamma_{h}(\mathcal{B}\bar{x}-c,v)(\mathcal{B}d)\Big\}>0,\,\forall\,d\neq0.
%	,\, \forall\, d\in\operatorname{aff}(\mathcal{C}_{1}(\bar{x}))\setminus \{0\}.
\end{align}
The SSOSC for the dual problem $(D')$ is
\begin{align}\label{eq:SSOSC-dual}
&\sup_{x\in\Lambda(\bar{u},\bar{v})}\Big\{\Gamma_{f^{*}}(\bar{u},\mathcal{A}x-b)(d_{1})+\Gamma_{h^{*}}(\bar{v},\mathcal{B}x-c)(d_{2})\Big\}>0,\,\forall\,(d_{1},d_{2})\neq0.
%	,\nonumber\\
%	&\qquad\qquad\qquad \forall\, (d_{1},d_{2})\in\operatorname{aff}(\mathcal{C}_{2}(\bar{u},\bar{v}))\setminus \{0\}.
\end{align}
Now we state the main result in the following theorem.
\begin{theorem}\label{thm:equiv-SSOSC-nondegeracy}
	Let $\bar{x}$ be an optimal solution of the primal problem $(P')$, $(\bar{u},\bar{v})$ be the corresponding solution of the dual problem $(D')$. One has
	\begin{enumerate}
		\item Suppose that $\Lambda(\bar{u},\bar{v})=\{\bar{x}\}$, the functions $f$ and $h$ satisfy Assumption \ref{assump-1}, Assumption \ref{assumption-2} at $\mathcal{A}\bar{x}-b$ for $\bar{u}$ and at $\mathcal{B}\bar{x}-c$ for $\bar{v}$, respectively. The nondegeneracy condition \eqref{eq:nondeg-cond-primal} of the primal problem $(P')$ holds at $\bar{x}$ if and only if the SSOSC \eqref{eq:SSOSC-dual} of the dual problem $(D')$ is valid at $(\bar{u},\bar{v})$.
		\item Suppose that $\Lambda(\bar{x})=\{(\bar{u},\bar{v})\}$, the functions $f^{*}$ and $h^{*}$ satisfy Assumption \ref{assump-1}, Assumption \ref{assumption-2} at $\bar{u}$ for $\mathcal{A}\bar{x}-b$ and at $\bar{v}$ for $\mathcal{B}\bar{x}-c$, respectively. The nondegeneracy condition \eqref{eq:nondeg-cond-dual} of the dual problem $(D')$ holds at $(\bar{u},\bar{v})$ if and only if the SSOSC \eqref{eq:SSOSC-primal} of the primal problem $(P')$ is valid at $\bar{x}$.
	\end{enumerate}
\end{theorem}

\begin{proof}
	We only prove the first part of the theorem and the second part can be proven similarly. Firstly, the nondegeneracy condition (\ref{eq:nondeg-cond-primal}) can be written equivalently as
	\begin{align}
	&\operatorname{NULL}([\mathcal{A}^{*}\ \mathcal{B}^{*}])\cap\left(\begin{array}{cc}\Big\{d\, \Big|\, f^{\downarrow}_{-}(\mathcal{A}\bar{x}-b,\mathcal{A}d)=-f^{\downarrow}_{-}(\mathcal{A}\bar{x}-b,-\mathcal{A}d)\Big\}^{\circ}\\
	\Big\{d\, \Big|\, h^{\downarrow}_{-}(\mathcal{B}\bar{x}-c,\mathcal{B}d)=-h^{\downarrow}_{-}(\mathcal{B}\bar{x}-c,-\mathcal{B}d)\Big\}^{\circ}\end{array}\right)=\{0\}.\label{eq:nondeg-cond-primal-equiv-1}
	\end{align}
	Based on Proposition 4.6 of \cite{TangW2024}, we know that the condition (\ref{eq:nondeg-cond-primal-equiv-1}) is  equivalent to
	\begin{align*}
	&\operatorname{NULL}([\mathcal{A}^{*}\ \mathcal{B}^{*}])\cap\\
	&\left(\begin{array}{cc}\operatorname{aff}\Big\{d_{1}\, \Big|\, f^{*\downarrow}(\bar{u},d_{1})=\langle \mathcal{A}\bar{x}-b,d_{1}\rangle\Big\}\cap\Big\{d_{1}\, \Big|\, \Gamma_{f^{*}}(\bar{u},\mathcal{A}\bar{x}-b)(d_{1})=0\Big\}\\
	\operatorname{aff}\Big\{d_{2}\, \Big|\, h^{*\downarrow}(\bar{v},d_{2})=\langle \mathcal{B}\bar{x}-c,d_{2}\rangle\Big\}\cap\Big\{d_{2}\, \Big|\, \Gamma_{h^{*}}(\bar{v},\mathcal{B}\bar{x}-c)(d_{2})=0\Big\}\end{array}\right)=\{0\},
	\end{align*}
	which can also be written as
	\begin{align}\label{eq:nondeg-cond-primal-equiv-2}
	&\operatorname{aff}(\mathcal{C}_{2}(\bar{u},\bar{v}))\\
	&\qquad\cap\Big\{(d_{1},d_{2})\, \Big|\, \Gamma_{f^{*}}(\bar{u},\mathcal{A}\bar{x}-b)(d_{1})=0, \Gamma_{h^{*}}(\bar{v},\mathcal{B}\bar{x}-c)(d_{2})=0\Big\}=\{0\}.\nonumber
	\end{align}
	Since the SSOSC condition (\ref{eq:SSOSC-dual}) is also equivalent to the condition (\ref{eq:nondeg-cond-primal-equiv-2}), the first part of the results holds.
\end{proof}

Based on Theorems 5.7 of \cite{TangW2024} and Theorem \ref{thm:equiv-SSOSC-nondegeracy}, we present a specific set of equivalent characterizations about the perturbation analysis in the following theorem.
\begin{theorem}
	Let $(\bar{x},\bar{u},\bar{v})$ be a solution of the KKT system \eqref{composite-problem-kkt}.
	Suppose the functions $f$, $h$, $f^{*}$ and $h^{*}$ satisfy Assumption \ref{assump-1}, Assumption \ref{assumption-2} and Assumption \ref{assumption-3} at $\mathcal{A}\bar{x}-b$ for $\bar{u}$, at $\mathcal{B}\bar{x}-c$ for $\bar{v}$, at $\bar{u}$ for $\mathcal{A}\bar{x}-b$ and at $\bar{v}$ for $\mathcal{B}\bar{x}-c$, respectively.
	Then the following statements are equivalent to each other.	
	\begin{enumerate}%[label=(\alph*)]
		\item The SSOSC \eqref{eq:SSOSC-primal} for the primal problem $(P')$ and the nondegeneracy condition \eqref{eq:nondeg-cond-primal} for the primal problem $(P')$ hold at $\bar{x}$.
		\item All of the elements of $\partial \widehat{R}(\bar{x},\bar{u},\bar{v})$ are nonsingular.
		\item The point $(\bar{x},\bar{u},\bar{v})$ is a strongly regular point of the KKT system \eqref{composite-problem-kkt}.
		\item $\widehat{R}$ is a locally Lipschitz homeomorphism near the KKT point $(\bar{x},\bar{u},\bar{v})$.
		%\item The nondegeneracy condition \eqref{eq:nondeg-cond-primal} for the primal problem $(P')$ holds and $Z$ has the Aubin property around $(0,\bar{x})$.
		%		\item The SSOSC \eqref{eq:SSOSC-dual} for the dual problem $(D')$ holds and $Z$ has the Aubin property around $(0,\bar{x})$.
		\item The nondegeneracy condition \eqref{eq:nondeg-cond-primal} for the primal problem $(P')$ and the nondegeneracy condition \eqref{eq:nondeg-cond-dual} for the dual problem $(D')$ hold.
		\item The SSOSC \eqref{eq:SSOSC-primal} for the primal problem $(P')$ and the SSOSC \eqref{eq:SSOSC-dual} for the dual problem $(D')$ are both valid.
		\item The nondegeneracy condition \eqref{eq:nondeg-cond-dual} for the dual problem $(D')$ and the SSOSC \eqref{eq:SSOSC-dual} for the dual problem $(D')$ are both valid.
	\end{enumerate}
\end{theorem}

\section{The SSN based ALM}
\label{sec:SSN-ALM}

In this section, we introduce the ALM for the convex composite optimization problem $(P')$ with the inner subproblem solved by the SSN method.

\subsection{The ALM for problem $(P')$}
\label{subsec:ALM}
Firstly, we introduce the details about the ALM. Given $\sigma>0$, we define the augmented Lagrangian function related to problem $(P')$ as follows
\begin{align*}
\mathcal{L}_{\sigma}(x;u,v)&:=\sup_{{s\in\mathcal{R}^{m_{1}}}\atop{t\in\mathcal{R}^{m_{2}}}}\Big\{L(x,s,t)-\frac{1}{2\sigma}\|s-u\|^{2}-\frac{1}{2\sigma}\|t-v\|^{2}\Big\}\\&=-\sigma^{-1}e_{\sigma f^{*}}(u+\sigma(\mathcal{A}x-b))+\langle u,\mathcal{A}x-b\rangle+\frac{\sigma}{2}\|\mathcal{A}x-b\|^{2}\\
&\quad-\sigma^{-1}e_{\sigma h^{*}}(v+\sigma(\mathcal{B}x-c))+\langle v,\mathcal{B}x-c\rangle+\frac{\sigma}{2}\|\mathcal{B}x-c\|^{2}.
\end{align*}
In the following, we introduce the ALM in Algorithm \ref{alg-ssnal}.
\begin{algorithm}[H]
	\caption{ALM}\label{alg-ssnal}
	Given $\sigma_{0}>0$, $\rho\geq1$ and $u^{0}\in\mathcal{R}^{m_{1}}$, $v^{0}\in\mathcal{R}^{m_{2}}$. For $k=0,1,2,\ldots$, iterate the following steps.
	\begin{description}
		\item[1] Find an approximate solution
		\begin{align}\label{SSNAL-subproblem}
		x^{k+1}\approx\bar{x}^{k+1}=\mathop{\operatorname{argmin}}_{x\in\mathcal{R}^{n}}\Big\{\Phi_{k}(x;u^{k},v^{k}):=\mathcal{L}_{\sigma_{k}}(x;u^{k},v^{k})\Big\}.
		\end{align}
		\item[2] Compute
		\begin{align*}
		u^{k+1}=u^{k}+\sigma_{k}(\mathcal{A}x^{k+1}-b),\quad v^{k+1}=v^{k}+\sigma_{k}(\mathcal{B}x^{k+1}-c),
		\end{align*}
		and update $\sigma_{k+1}=\rho\sigma_{k}$.
	\end{description}
\end{algorithm}

For most cases, the inner subproblem \eqref{SSNAL-subproblem} has no closed-form solution and we can only find an approximate one with the following stopping criteria introduced in \cite{Rockafellar1976-1,Rockafellar1976-2}.
\begin{eqnarray*}
&(A)& \Phi_{k}(x^{k+1};u^{k},v^{k})-\inf_{x\in\mathcal{R}^{n}}\Phi_{k}(x;u^{k},v^{k})\leq\frac{\varepsilon_{k}^{2}}{2\sigma_{k}},\ \varepsilon_{k}\geq 0,\ \sum_{k=0}^{\infty}\varepsilon_{k}<\infty,\\
&(B)& \Phi_{k}(x^{k+1};u^{k},v^{k})-\inf_{x\in\mathcal{R}^{n}}\Phi_{k}(x;u^{k},v^{k})\leq\frac{\delta_{k}^{2}}{2\sigma_{k}}\|(u^{k+1},v^{k+1})-(u^{k},v^{k})\|^{2},\\
&&\qquad\qquad\qquad\qquad\qquad\qquad\qquad\qquad\qquad\qquad 0\leq\delta_{k}<1,\ \sum_{k=0}^{\infty}\delta_{k}<\infty,\\
&(B')&\quad\|\nabla\Phi_{k}(x^{k+1};u^{k},v^{k})\|\leq\frac{\delta'_{k}}{\sigma_{k}}\|(u^{k+1},v^{k+1})-(u^{k},v^{k})\|,\ 0\leq\delta'_{k}\rightarrow 0.
\end{eqnarray*}
The convergence result for the ALM follows easily from Theorems 1 of \cite{Rockafellar1976-1}, Theorem 4.2 of \cite{Luque1984} and Theorem 4 of \cite{Rockafellar1976-2}.
\begin{theorem}	
	Suppose that $\inf(P')<\infty$ and the ALM is performed with the stopping criterion $(A)$. If the sequence $\{(x^{k},u^{k},v^{k})\}$ generated by the ALM is bounded, then $x^{k}$ converges to $\bar{x}$, which is an optimal solution of problem $(P')$, and $(u^{k},v^{k})$ converges to an optimal solution of problem $(D')$.
\end{theorem}	

Let $D:=\operatorname{NULL}([\mathcal{A}^*\ \mathcal{B}^{*}])$ and $H(u,v):=f^{*}(u)+h^{*}(v)+\langle b,u\rangle+\langle c,v\rangle+\delta_{D}(u,v)$. Denote
\begin{align*}
&T(u,v):=\partial H(u,v),\quad T_{\widehat{g}}(x):=\partial (g\circ F)(x),\\
&T_{l}(x,u,v):=\Big\{(x',u',v')\,\Big|\, (x',-u',-v')\in\partial L(x,u,v)\Big\}.
\end{align*}
We state the convergence rate result for the ALM in the following theorem. One may also see \cite{CuiSun2019} for more details.
\begin{theorem}\label{theorem-alm-conv-rate}
	Suppose that $\inf(P')<\infty$ and the sequence $\{(x^{k},u^{k},v^{k})\}$ is generated by the ALM under the stopping criterion $(A)$ with $(u^{k},v^{k})$ converging to $(\bar{u},\bar{v})$, which is a solution of the dual problem $(D')$. If $T^{-1}$ is calm at the origin for $(\bar{u},\bar{v})$ with the modulus $a_{g}$. Then under the criterion $(B)$, for $k$ sufficiently large, we have that
	\begin{align*}
	\operatorname{dist}((u^{k+1},v^{k+1}),T^{-1}(0))\leq\theta_{k}\operatorname{dist}((u^{k},v^{k}),T^{-1}(0)),
	\end{align*}
	where
	\begin{align*}
	&\theta_{k}=\left[(\delta_{k}+1)a_{g}(a_{g}^{2}+\sigma_{k}^{2})^{-1/2}+\delta_{k}\right](1-\delta_{k})^{-1}\rightarrow\theta_{\infty}=a_{g}(a_{g}^{2}+\sigma_{\infty}^{2})^{-1/2}<1,\\
	&\theta_{\infty}=0,\quad\mbox{if}\quad \sigma_{\infty}=+\infty.
	\end{align*}
	Furthermore, if one also has the criterion $(B')$ and there exist positive constants $\varepsilon$ and $a_{l}$ such that
	\begin{align*}
	&\operatorname{dist}((x,u,v),T_{l}^{-1}(0))\leq a_{l}\|(x',u',v')\|,\\
	&\qquad\qquad\qquad\qquad\forall\, (x,u,v)\in T_{l}^{-1}(x',u',v'),\,(x',u',v')\in\mathbb{B}(0,\varepsilon),
	\end{align*}
	then for $k$ sufficiently large, one also has that
	\begin{align*}
	\operatorname{dist}(x^{k+1},T_{\widehat{g}}^{-1}(0))\leq\theta'_{k}\|(u^{k+1},v^{k+1})-(u^{k},v^{k})\|,
	\end{align*}
	where
	\begin{align*}
	&\theta'_{k}=a_{l}/\sigma_{k}(1+\delta_{k}^{\prime2})\rightarrow\theta'_{\infty}=a_{l}/\sigma_{\infty},\quad
	\theta'_{\infty}=0,\quad\mbox{if}\quad \sigma_{\infty}=+\infty.
	\end{align*}
\end{theorem}

As we know from 3H of \cite{DontchevandRockafellar2009} that the metric subregularity is equivalent to the calmness property. Furthermore, it is also known from \cite{Francisco2008} that the metric subregularity of the dual solution mapping at a dual optimal solution for the origin is equivalent to the quadratic growth condition at the corresponding dual optimal solution. Combining Proposition 3 of \cite{CuiSun2019}, it follows that we need the quadratic growth condition to guarantee the superlinear convergence rate of the dual sequence generated by the ALM. However, it is usually difficult to check the quadratic growth condition directly and one may deal with it alternatively by looking for other sufficient conditions. One of such conditions is the second order sufficient condition (SOSC). As shown in Theorem 3.109 of \cite{Bonnans2000}, under the conditions that $g^{*}$ is a lower semicontinuous proper convex function and the set $\operatorname{epi}g^{*}$ is outer second order regular, the SOSC implies the quadratic growth condition.

In the follows, we introduce the equivalence of the SRCQ of the primal/dual problem and the SOSC of the dual (primal) problem. We first write out the SRCQ and SOSC for the primal and dual problems, respectively.

Let $\bar{x}$ be an optimal solution of the primal problem $(P')$ and $(\bar{u},\bar{v})$ be a corresponding solution of the dual problem $(D')$. Then the SRCQ for the primal problem $(P')$ at $\bar{x}$ for $(\bar{u},\bar{v})$ and the SRCQ for the dual problem $(D')$ at $(\bar{u},\bar{v})$ for $\bar{x}$ are
\begin{align}
&\left(\begin{array}{cc}\mathcal{A}\\\mathcal{B}\end{array}\right)\mathcal{R}^{n}-\left(\begin{array}{cc}\Big\{d\, \Big|\, f^{\downarrow}_{-}(\mathcal{A}\bar{x}-b,\mathcal{A}d)=\langle\bar{u},d\rangle\Big\}\\
\Big\{d\, \Big|\, h^{\downarrow}_{-}(\mathcal{B}\bar{x}-c,\mathcal{B}d)=\langle\bar{v},d\rangle\Big\}\end{array}\right)=\left(\begin{array}{cc}\mathcal{R}^{m_{1}}\\\mathcal{R}^{m_{2}}\end{array}\right)\label{SRCQ-primal}
\end{align}	
and
\begin{align}
&\mathcal{A}^{*}\Big\{d_{1}\, \Big|\, f^{*\downarrow}_{-}(\bar{u},d_{1})=\langle\mathcal{A}\bar{x}-b,d_{1}\rangle\Big\}+\mathcal{B}^{*}\Big\{d_{2}\, \Big|\, h^{*\downarrow}_{-}(\bar{v},d_{2})=\langle\mathcal{B}\bar{x}-c,d_{2}\rangle\Big\}=\mathcal{R}^{n},\label{SRCQ-dual}
\end{align}
respectively. Suppose that $f$ and $h$ are outer second order regular. We can also write out the SOSC for the primal problem $(P')$ at $\bar{x}$ for $(\bar{u},\bar{v})$ as
\begin{align}
\varphi_{f}(\mathcal{A}\bar{x}-b,\bar{u})(\mathcal{A}d)+\varphi_{h}(\mathcal{B}\bar{x}-c,\bar{v})(\mathcal{B}d)>0,\, \forall\, d\in\mathcal{C}_{1}(\bar{x})\setminus\{0\},\label{SOSC-primal}
\end{align}
where
\begin{align*}
\mathcal{C}_{1}(\bar{x})&=\Big\{d\, \Big|\, f^{\downarrow}_{-}(\mathcal{A}\bar{x}-b,\mathcal{A}d)+h^{\downarrow}_{-}(\mathcal{B}\bar{x}-c,\mathcal{B}d)=0\Big\}\\
&=\Big\{d\, \Big|\, f^{\downarrow}_{-}(\mathcal{A}\bar{x}-b,\mathcal{A}d)=\langle \bar{u},\mathcal{A}d \rangle, h^{\downarrow}_{-}(\mathcal{B}\bar{x}-c,\mathcal{B}d)=\langle \bar{v},\mathcal{B}d \rangle\Big\}
\end{align*}
is the critical cone of the primal problem $(P')$ at $\bar{x}$. Suppose that $f^{*}$ and $h^{*}$ are outer second order regular. The SOSC for the dual problem $(D')$ at $(\bar{u},\bar{v})$ for $\bar{x}$ takes the following form
\begin{align}
\varphi_{f^{*}}(\bar{u},\mathcal{A}\bar{x}-b)(d_{1})+\varphi_{h^{*}}(\bar{v},\mathcal{B}\bar{x}-c)(d_{2})>0,\, \forall\, (d_{1},d_{2})\in\mathcal{C}_{2}(\bar{u},\bar{v})\setminus\{0\},\label{SOSC-dual}
\end{align}
where
\begin{align*}
&\mathcal{C}_{2}(\bar{u},\bar{v})\\
&=\Big\{(d_{1},d_{2})\, \Big|\, f^{*\downarrow}_{-}(\bar{u},d_{1})+\langle b,d_{1}\rangle+h^{*\downarrow}_{-}(\bar{v},d_{2})+\langle c,d_{2}\rangle=0,\mathcal{A}^{*}d_{1}+\mathcal{B}^{*}d_{2}=0\Big\}\\
&=\Big\{(d_{1},d_{2})\, \Big|\, f^{*\downarrow}_{-}(\bar{u},d_{1})=\langle \mathcal{A}\bar{x}-b,d_{1} \rangle,  h^{*\downarrow}_{-}(\bar{v},d_{2})=\langle \mathcal{B}\bar{x}-c,d_{2} \rangle,  \mathcal{A}^{*}d_{1}+\mathcal{B}^{*}d_{2}=0\Big\}
\end{align*}
is the critical cone of the dual problem $(D')$ at $(\bar{u},\bar{v})$.

Now we are ready to characterize the equivalence of the SRCQ and the SOSC.
\begin{theorem}\label{thm-1}
	For problem $(P')$, we have
	\begin{enumerate}	
		\item If $f^{*}$ and $h^{*}$ are $C^{2}$-cone reducible functions, the SRCQ at $(\bar{u},\bar{v})$ for $\bar{x}$ of the dual problem $(D')$ holds if and only if the SOSC at $\bar{x}$ for $(\bar{u},\bar{v})$ of the primal problem $(P')$ is valid.
		\item If $f$ and $h$ are $C^{2}$-cone reducible functions, the SRCQ at $\bar{x}$ for $(\bar{u},\bar{v})$ of the primal problem $(P')$ holds if and only if the SOSC at $(\bar{u},\bar{v})$ for $\bar{x}$ of the dual problem $(D')$ is valid,
	\end{enumerate}
\end{theorem}
\begin{proof}
	We first prove that the SRCQ \eqref{SRCQ-dual} is equivalent to the SOSC (\ref{SOSC-primal}). Taking the polar operation on both sides of (\ref{SRCQ-dual}), by Corollary 16.4.2 of \cite{Rockafellar1970} the SRCQ (\ref{SRCQ-dual}) can be written equivalently as
	that
	\begin{align}\label{eq:SRCQ-dual-equiv-1}
	&\mathcal{A}d\in\Big\{d_{1}\, \Big|\, f^{*\downarrow}_{-}(\bar{u},d_{1})=\langle\mathcal{A}\bar{x}-b,d_{1}\rangle\Big\}^{\circ},\  \mathcal{B}d\in\Big\{d_{2}\, \Big|\, h^{*\downarrow}_{-}(\bar{v},d_{2})=\langle\mathcal{B}\bar{x}-c,d_{2}\rangle\Big\}^{\circ}\nonumber\\
	&\qquad\qquad\qquad\qquad\qquad\Longrightarrow d=0.
	\end{align}
	Due to Proposition 3.4 of \cite{TangW2024} and the fact that $f$ and $h$ are outer second order regular at $\mathcal{A}\bar{x}-b$ and $\mathcal{B}\bar{x}-c$, respectively, we know that the condition (\ref{eq:SRCQ-dual-equiv-1}) is also equivalent to
	\begin{align}
	&\Big\{d\, \Big|\, f^{\downarrow}_{-}(\mathcal{A}\bar{x}-b,\mathcal{A}d)=\langle\bar{u},\mathcal{A}d\rangle,  \varphi_{f}(\mathcal{A}\bar{x}-b,\bar{u})(\mathcal{A}d)=0\Big\}\nonumber\\
	&\quad\cap\Big\{d\, \Big|\, h^{\downarrow}_{-}(\mathcal{B}\bar{x}-c,\mathcal{B}d)=\langle\bar{v},\mathcal{B}d\rangle,  \varphi_{h}(\mathcal{B}\bar{x}-c,\bar{v})(\mathcal{B}d)=0\Big\}=\{0\}.\label{eq:SRCQ-dual-equiv-2}
	\end{align}
	Due to the definition of $\mathcal{C}_{1}(\bar{x})$, the condition (\ref{eq:SRCQ-dual-equiv-2}) is valid if and only if
	\begin{align}
	\mathcal{C}_{1}(\bar{x})\cap\Big\{d\, \Big|\, \varphi_{f}(\mathcal{A}\bar{x}-b,\bar{u})(\mathcal{A}d)=0,  \varphi_{h}(\mathcal{B}\bar{x}-c,\bar{v})(\mathcal{B}d)=0\Big\}=\{0\}\label{eq:SRCQ-dual-equiv-3}
	\end{align}
	holds. Since $\varphi_{f}(\mathcal{A}\bar{x}-b,\bar{u})(\mathcal{A}d)\geq0$ and $\varphi_{h}(\mathcal{B}\bar{x}-c,\bar{v})(\mathcal{B}d)\geq0$, the SOSC condition (\ref{SOSC-primal}) is also equivalent to (\ref{eq:SRCQ-dual-equiv-3}). Therefore, the first part of the results can be obtained easily.
	
	Similarly, the SRCQ \eqref{SRCQ-primal} holds if and only if
	\begin{align}
	\operatorname{NULL}([\mathcal{A}^{*}\ \mathcal{B}^{*}]) \cap
	\left(\begin{array}{cc}\Big\{d\,\Big|\, f^{\downarrow}_{-}(\mathcal{A}\bar{x}-b,\mathcal{A}d)=\langle\bar{u},d\rangle\Big\}^{\circ}\\
	\Big\{d\,\Big|\, h^{\downarrow}_{-}(\mathcal{B}\bar{x}-c,\mathcal{B}d)=\langle\bar{v},d\rangle\Big\}^{\circ}\end{array}\right)=\{0\}
	\label{eq:SRCQ-equiv-1}
	\end{align}
	is valid. Due to Proposition 3.4 of \cite{TangW2024}, the condition (\ref{eq:SRCQ-equiv-1}) is equivalent to
	\begin{align}
	&\operatorname{NULL}([\mathcal{A}^{*}\ \mathcal{B}^{*}]\nonumber\\
	&\quad\cap\left(\begin{array}{cc}\Big\{d_{1}\, \Big|\, f^{*\downarrow}_{-}(\bar{u},d_{1})=\langle \mathcal{A}\bar{x}-b,d_{1}\rangle, \varphi_{f^{*}}(\bar{u},\mathcal{A}\bar{x}-b)(d_{1})=0\Big\}\\
	\Big\{d_{2}\, \Big|\, h^{*\downarrow}_{-}(\bar{v},d_{2})=\langle \mathcal{B}\bar{x}-c,d_{2}\rangle, \varphi_{h^{*}}(\bar{v},\mathcal{B}\bar{x}-c)(d_{2})=0\Big\}\end{array}\right)=\{0\}.\label{eq:SRCQ-equiv-2}
	\end{align}
	Based on the definition of $\mathcal{C}_{2}(\bar{u},\bar{v})$, the condition (\ref{eq:SRCQ-equiv-2})  is valid if and only if
	\begin{align}
	\mathcal{C}_{2}(\bar{u},\bar{v})\cap\Big\{(d_{1},d_{2})\, \Big|\, \varphi_{f^{*}}(\bar{u},\mathcal{A}\bar{x}-b)(d_{1})=0,  \varphi_{h^{*}}(\bar{v},\mathcal{B}\bar{x}-c)(d_{2})=0\Big\}=\{0\}\label{eq:SRCQ-equiv-3}
	\end{align}
	holds. Since $\varphi_{f^{*}}(\bar{u},\mathcal{A}\bar{x}-b)(d_{1})\geq 0$ and $\varphi_{h^{*}}(\bar{v},\mathcal{B}\bar{x}-c)(d_{2})\geq 0$, the SOSC condition (\ref{SOSC-dual}) is also equivalent to (\ref{eq:SRCQ-equiv-3}). Therefore, the second part of the results can also be obtained and this completes the proof.
	
\end{proof}

\subsection{The SSN for the inner subproblem \eqref{SSNAL-subproblem}}
\label{subsec:SSN}
The main challenge of the ALM is that we need to find an approximate solution with the given accuracy for the inner subproblem efficiently. Due to the continuous differentiability of the Moreau envelope functions for $f$ and $g$, the function $\Phi_{k}(\cdot;u^{k},v^{k})$ is smooth and solving the subproblem (\ref{SSNAL-subproblem}) is equivalent to finding a solution of the following system of equations
\begin{align}\label{eq:subproblem-nonlinear-eq}
\nabla\Phi_{k}(x;u^{k},v^{k})&=\mathcal{A}^{*}\operatorname{Prox}_{\sigma_{k} f^{*}}(u^{k}+\sigma_{k}(\mathcal{A}x-b))\nonumber\\
&\qquad+\mathcal{B}^{*}\operatorname{Prox}_{\sigma_{k} h^{*}}(v^{k}+\sigma_{k}(\mathcal{B}x-c))=0.
\end{align}
We apply the SSN method to obtain an approximate solution of problem (\ref{eq:subproblem-nonlinear-eq}). Since $\operatorname{Prox}_{\sigma f^{*}}$ and $\operatorname{Prox}_{\sigma g^{*}}$ are Lipschitz continuous, the following multifunction is well defined.
\begin{align*}
\hat{\partial}^{2}\Phi_{k}(x;u^{k},v^{k})&:=\sigma_{k}\mathcal{A}^{*}\partial\operatorname{Prox}_{\sigma_{k} f^{*}}(u^{k}+\sigma_{k}(\mathcal{A}x-b))\mathcal{A}\\
&\qquad+\sigma_{k}\mathcal{B}^{*}\partial\operatorname{Prox}_{\sigma_{k} h^{*}}(v^{k}+\sigma_{k}(\mathcal{B}x-c))\mathcal{B}.
\end{align*}
Let $U^{k}\in\partial\operatorname{Prox}_{\sigma_{k} f^{*}}(u^{k}+\sigma_{k}(\mathcal{A}x-b))$ and $V^{k}\in\partial\operatorname{Prox}_{\sigma_{k} h^{*}}(v^{k}+\sigma_{k}(\mathcal{B}x-c))$, we have $\sigma_{k}\mathcal{A}^{*}U^{k}\mathcal{A}+\sigma_{k}\mathcal{B}^{*}V^{k}\mathcal{B}\in\hat{\partial}^{2}\Phi_{k}(x;u^{k},v^{k})$.

For simplicity, we omit the subscripts or superscripts of $\sigma_{k}$, $\Phi_{k}$, $U^{k}$ and $V^{k}$. Let $(\bar{x},\bar{u},\bar{v})$ be a solution of the KKT system \eqref{composite-problem-kkt}.  The following proposition presents an equivalent condition to guarantee that each element of the set $\hat{\partial}^{2}\Phi(\bar{x};\bar{u},\bar{v})$ is nonsingular.

\begin{theorem}\label{theorem-nonsingularity}
	Let $(\bar{x},\bar{u},\bar{v})$ be a solution of the KKT system \eqref{composite-problem-kkt}.
	Suppose that the functions $f^{*}$ and $h^{*}$ satisfy Assumption \ref{assump-1}, Assumption \ref{assumption-2} at $\bar{u}$ for $\mathcal{A}\bar{x}-b$ and at $\bar{v}$ for $\mathcal{B}\bar{x}-c$, respectively.
	Then the nondegeneracy condition (\ref{eq:nondeg-cond-dual}) holds if and only if all the elements of $\hat{\partial}^{2}\Phi(\bar{x};\bar{u},\bar{v})$ are nonsingular.
\end{theorem}
\begin{proof}
	Let $E:=\sigma\mathcal{A}^{*}U\mathcal{A}+\sigma\mathcal{B}^{*}V\mathcal{B}$ be an arbitrary element of $\hat{\partial}^{2}\Phi(\bar{x};\bar{u},\bar{v})$ with $U\in\partial\operatorname{Prox}_{\sigma f^{*}}(\bar{u}+\sigma(\mathcal{A}\bar{x}-b))$ and $V\in\partial\operatorname{Prox}_{\sigma h^{*}}(\bar{v}+\sigma(\mathcal{B}\bar{x}-c))$.
	
	On one hand, suppose that $d$ satisfies $\langle d,Ed\rangle=0$. Since $U$ and $V$ are positively semidefinite, we have $\langle \mathcal{A}d,U\mathcal{A}d\rangle=0$ and $\langle \mathcal{B}d,V\mathcal{B}d\rangle=0$, which is also equivalent to $U(\mathcal{A}d)=0$ and $V(\mathcal{B}d)=0$. Given $y$ and $z$, let $f_{\sigma}(y):=(\sigma f^{*})^{*}(y)=\sigma f(\sigma^{-1}y)$ and $h_{\sigma}(z):=(\sigma h^{*})^{*}(z)=\sigma h(\sigma^{-1}z)$.
	By Lemma 4.2 of \cite{TangW2024}, it follows that
	\begin{align*}
	& \Gamma_{f_{\sigma}}(\sigma(\mathcal{A}\bar{x}-b),\bar{u})(\mathcal{A}d)=0,\quad \mathcal{A}d\in\operatorname{aff}\left(\Big\{d\, \Big|\, (f_{\sigma})^{\downarrow}_{-}(\sigma(\mathcal{A}\bar{x}-b),d)=\langle\bar{u},d\rangle\Big\}\right),\\
	&\Gamma_{h_{\sigma}}(\sigma(\mathcal{B}\bar{x}-c),\bar{v})(\mathcal{B}d)=0,\quad \mathcal{B}d\in\operatorname{aff}\left(\Big\{d\, \Big|\, (h_{\sigma})^{\downarrow}_{-}(\sigma(\mathcal{B}\bar{x}-b),d)=\langle\bar{v},d\rangle\Big\}\right).
	\end{align*}
	Therefore, due to Proposition 4.6 of \cite{TangW2024}, we can obtain
	\begin{align}\label{eq:prop10-1}
	\left.\begin{array}{cc}\mathcal{A}d\in\Big\{d_{1}\, \Big|\, (\sigma f^{*})^{\downarrow}_{-}(\bar{u},d_{1})=-(\sigma f^{*})^{\downarrow}_{-}(\bar{u},-d_{1})\Big\}^{\circ},\\
	\mathcal{B}d\in\Big\{d_{2}\, \Big|\, (\sigma h^{*})^{\downarrow}_{-}(\bar{v},d_{2})=-(\sigma h^{*})^{\downarrow}_{-}(\bar{v},-d_{2})\Big\}^{\circ}, \end{array}\right.
	\end{align}
	which is equivalent to
	\begin{align*}
	d\in\Big(\mathcal{A}^{*}\Big\{d_{1}\, \Big|\, f^{*\downarrow} _{-}(\bar{u},d_{1})=-f^{*\downarrow}(\bar{u},-d_{1})\Big\}\Big)^{\circ}\cap\Big(\mathcal{B}^{*}\Big\{d_{2}\, \Big|\, h^{*\downarrow} _{-}(\bar{v},d_{2})=-h^{*\downarrow}(\bar{v},-d_{2})\Big\}\Big)^{\circ}.
	\end{align*}
	By the nondegeneracy condition \eqref{eq:nondeg-cond-dual} we obtain $d=0$ and then $E$ is nonsingular.
	
	On the other hand, suppose that all the elements of $\hat{\partial}^{2}\Phi(\bar{x};\bar{\mu})$ are nonsingular and the nondegeneracy condition \eqref{eq:nondeg-cond-dual} is invalid. Then there exists $d\neq 0$ and \eqref{eq:prop10-1} holds. By Proposition 4.6 and Lemma 4.2 of \cite{TangW2024}, there exist $U\in\partial\operatorname{Prox}_{\sigma f^{*}}(\bar{u}+\sigma(\mathcal{A}\bar{x}-b))$ and $V\in\partial\operatorname{Prox}_{\sigma h^{*}}(\bar{v}+\sigma(\mathcal{B}\bar{x}-c))$ such that $U(\mathcal{A}d)=0$ and $V(\mathcal{B}d)=0$, which is a contradiction with the nonsingularity of all the elements of 	$\hat{\partial}^{2}\Phi(\bar{x};\bar{u},\bar{v})$. This completes the proof.
\end{proof}		

Now we state the SSN method in Algorithm \ref{alg-ssn}.
\begin{algorithm}
	\caption{SSN}\label{alg-ssn}
	Input $\sigma>0$, $\tilde{u}\in\mathcal{R}^{m_{1}}$, $\tilde{v}\in\mathcal{R}^{m_{2}}$
	$\nu\in(0,\frac{1}{2})$,  $\nu_{1},\nu_{2},\bar{\eta}\in(0,1)$,
	$\tau\in(0,1]$ and $\delta\in(0,1)$. Choose $x^{0}\in\mathcal{R}^{n}$. Set $j=0$ and iterate:
	\begin{description}
		\item [Step 1] Let $U^{j}\in\partial\operatorname{Prox}_{\sigma f^{*}}(\sigma(\mathcal{A}x^{j}-b)+\tilde{u})$, $V^{j}\in\partial\operatorname{Prox}_{\sigma h^{*}}(\sigma(\mathcal{B}x^{j}-c)+\tilde{v})$, and $H^{j}=\sigma \mathcal{A}^{*}U^{j}\mathcal{A}+\sigma \mathcal{B}^{*}V^{j}\mathcal{B}$. Solve the following linear system
		\begin{eqnarray*}
			(H^{j}+\epsilon_{j}I)\Delta x=-\nabla\Phi(x^{j};\tilde{u},\tilde{v})
		\end{eqnarray*}
		by a direct method or the preconditioned conjugate gradient method to obtain an approximate solution $\Delta x^{j}$ satisfying the condition below
		\begin{eqnarray*}
			\|(H^{j}+\epsilon_{j}I)\Delta x^{j}+\nabla\Phi(x^{j};\tilde{u},\tilde{v})\|\leq
			\eta_{j}:=\min(\bar{\eta},\|\nabla\Phi(x^{j};\tilde{u},\tilde{v})\|^{1+\tau}),%\label{ineq:stopping criterion 1}
		\end{eqnarray*}
		where $\epsilon_{j}=\nu_{1}\min\{\nu_{2},\|\nabla\Phi(x^{j};\tilde{u},\tilde{v})\|\}$.
		\item [Step 2]  Set $\alpha_{j}=\delta^{m_{j}}$, where $m_{j}$ is the first nonnegative integer $m$ such that
		\begin{eqnarray*}
			&\Phi(x^{j}+\delta^{m}\Delta x^{j};\tilde{u},\tilde{v})\leq
			\Phi(x^{j};\tilde{u},\tilde{v})+\nu\delta^{m}\langle\nabla
			\Phi(x^{j};\tilde{u},\tilde{v}),\Delta x^{j}\rangle.&
		\end{eqnarray*}
		\item [Step 3]  Set $x^{j+1}=x^{j}+\alpha_{j}\Delta x^{j}$. If a desired stopping criterion is satisfied, terminate; otherwise set $j=j+1$ and go to Step 1.
	\end{description}
\end{algorithm}
We can mimic the proof of Theorem 3.5 in \cite{Zhaoxy2010} and obtain the convergence result for the SSN method easily. We just omit the details of the proof and list the result in the following theorem. One may see \cite{Zhaoxy2010} for more details.
\begin{theorem}
	Assume that the functions $f^{*}$ and $h^{*}$ satisfy Assumption \ref{assump-1}, Assumption \ref{assumption-2} at $\bar{u}$ for $\mathcal{A}\bar{x}-b$ and at $\bar{v}$ for $\mathcal{B}\bar{x}-c$, respectively.
	Suppose that the nondegeneracy condition \eqref{eq:nondeg-cond-dual} holds and $\operatorname{Prox}_{f^{*}}$ and $\operatorname{Prox}_{g^{*}}$ are semismooth everywhere.
	Given $\sigma>0$, $x^{0}\in\mathcal{R}^{n}$, $\tilde{u}\in\mathcal{R}^{m_{1}}$ and $\tilde{v}\in\mathcal{R}^{m_{2}}$, the sequence $\{x^{j}\}$ generated by the SSN method converges to the unique solution $\bar{x}$ of $\nabla\Phi(x;\tilde{u},\tilde{v})=0$ and it holds that
	\begin{eqnarray*}
		\|x^{j+1}-\bar{x}\|=o(\|x^{j}-\bar{x}\|).
	\end{eqnarray*}
	If in addition $\operatorname{Prox}_{f^{*}}$ and $\operatorname{Prox}_{g^{*}}$ are $p$-order semismooth everywhere, the convergence rate of the sequence $\{x^{j}\}$ is of order $1+\min\{p,\tau\}$.
\end{theorem}	

\section{The von Neumann entropy optimization problem}
\label{sec:entropy}
In this subsection, we focus on applying the ALM to the following von Neumann entropy optimization problem
\begin{eqnarray}
\label{eq:entropy-primal-problem}
\min_{X}\left\{\inprod{C}{X}+\mu\trace(X\log X)\,|\,\mathcal{A}(X)=b,\,\mathcal{B}(X)\geq d,\,X-\varepsilon I\succeq 0\,\right\},
\end{eqnarray}
where $\mu>0$, $C\in \mathcal{S}^{n}$, $b\in\mathcal{R}^{m_1}$, $d\in\mathcal{R}^{m_2}$, $\mathcal{A}: \mathcal{S}^{n}\rightarrow\mathcal{R}^{m_1}$, $\mathcal{B}: \mathcal{S}^{n}\rightarrow\mathcal{R}^{m_2}$.

The Lagrangian function of problem \eqref{eq:entropy-primal-problem} is
\begin{eqnarray*}
	\label{eq:Lagrangian-entropy-primal-problem}
	L(X;y,z,S) &=& \inprod{C}{X} + \mu\trace(X\log X) - \inprod{y}{\mathcal{A}(X)-b} - \inprod{z}{\mathcal{B}(X)-d} - \\
	&& \inprod{S}{X-\varepsilon I}.
\end{eqnarray*}

The KKT condition of problem \eqref{eq:entropy-primal-problem} is
\begin{eqnarray*}
	\label{eq:KKT-entropy-primal-problem}
	\begin{array}{lll}
		&C-\mathcal{A}^*y-\mathcal{B}^*z-S+\mu\log X + \mu I = 0,&\\
		&\mathcal{A}(X)-b=0,\,\mathcal{B}(X)-d\geq 0,\,X-\varepsilon I\succeq 0,&\\
		&z\geq 0,\,S\succeq 0,\,\inprod{\mathcal{B}(X)-d}{z}=0,\,\inprod{X-\varepsilon I}{S}=0.&\\
	\end{array}
\end{eqnarray*}

The augmented Lagrangian function of problem \eqref{eq:entropy-primal-problem} is
\begin{eqnarray*}
	\label{eq:AL-fun-entropy-primal-problem}
	L_{\sigma}(X;y,z,S) &=& \inprod{C}{X} + \mu\trace(X\log X) + \frac{\sigma}{2}\norm{\mathcal{A}(X)-b-\frac{1}{\sigma}y}^{2} - \frac{1}{2\sigma}\norm{y}^{2} + \\ &&\frac{1}{2\sigma}\norm{\mathrm{\Pi}_{\mathcal{R}^{m_2}_{+}}(z-\sigma(\mathcal{B}(X)-d))}^{2} - \frac{1}{2\sigma}\norm{z}^{2} + \\ &&\frac{1}{2\sigma}\norm{\mathrm{\Pi}_{\mathcal{S}^{n}_{+}}(S-\sigma(X-\varepsilon I))}^{2} - \frac{1}{2\sigma}\norm{S}^{2}.
\end{eqnarray*}
Then we can calculate
\begin{eqnarray*}
	\label{eq:grad-AL-fun-entropy-primal-problem}
	\nabla L_{\sigma}(X;y,z,S) &=& C + \mu\log X +\mu I +\sigma\mathcal{A}^*(\mathcal{A}(X)-b-\frac{1}{\sigma}y) - \\
	&&\mathcal{B}^*\mathrm{\Pi}_{\mathcal{R}^{m_2}_{+}}(z-\sigma(\mathcal{B}(X)-d)) - \mathrm{\Pi}_{\mathcal{S}^{n}_{+}}(S-\sigma(X-\varepsilon I)).
\end{eqnarray*}
Further, we define
\begin{eqnarray*}
	\label{eq:Hess-AL-fun-entropy-primal-problem}
	\widehat{\partial}^{2} L_{\sigma}(X;y,z,S) &=& (\mu\log)' (X)  + \sigma\mathcal{A}^*\mathcal{A} + \sigma\mathcal{B}^*\partial\mathrm{\Pi}_{\mathcal{R}^{m_2}_{+}}(z-\sigma(\mathcal{B}(X)-d))\mathcal{B} + \\
	&&\sigma \partial\mathrm{\Pi}_{\mathcal{S}^{n}_{+}}(S-\sigma(X-\varepsilon I)).
\end{eqnarray*}

Now we present the details of the ALM for problem \eqref{eq:entropy-primal-problem}.
\begin{algorithm}
	\caption{ALM for problem \eqref{eq:entropy-primal-problem}}\label{alg-ssnal-entropy}
	Given $\sigma_{0}>0$, $\rho\geq1$ and $X^{0},S^{0}\in\mathcal{S}^{n}$, $y^{0}\in\mathcal{R}^{m_{1}}$, $z^{0}\in\mathcal{R}^{m_{2}}$. For $k=0,1,2,\ldots$, iterate the following steps.
	\begin{description}
		\item[1] Find an approximate solution
		\begin{align}\label{SSNAL-subproblem-entropy}
		X^{k+1}\approx\overline{X}^{k+1}=\mathop{\operatorname{argmin}}_{X\in\mathcal{S}^{n}}L_{\sigma_k}(X;y^{k},z^{k},S^{k}).
		\end{align}
		\item[2] Compute
		\begin{eqnarray*}
			y^{k+1}&=&y^{k}-\sigma_{k}(\mathcal{A}(X^{k+1})-b),\\
			z^{k+1}&=&\mathrm{\Pi}_{\mathcal{R}^{m_2}_{+}}(z^{k}-\sigma_{k}(\mathcal{B}(X^{k+1})-d)),\\
			S^{k+1}&=&\mathrm{\Pi}_{\mathcal{S}^{n}_{+}}(S^{k}-\sigma(X^{k+1}-\varepsilon I)).
		\end{eqnarray*}
		and update $\sigma_{k+1}=\rho\sigma_{k}$.
	\end{description}
\end{algorithm}

Since the subproblem \eqref{SSNAL-subproblem-entropy} often has no analytical solution,
the stopping criteria for the subproblem in the algorithm needs to be discussed. For $X_P := X \in \mathcal{S}^{n}$ and $X_D :=(y, z, S) \in  \mathcal{R}^{m_1} \times \mathcal{R}^{m_2}_{+}\times \mathcal{S}^{n}_{+}$, at the $k$th iteration, we define functions $f_k$ and $g_k$ as follows:
\begin{equation*}
\begin{aligned}
f_k(X_D) &:= L_{\sigma_k}(X_P; X_D^k) = \sup_{X_D} \Big\{L(X_P;  X_D) - \frac{1}{2\sigma_k}\|X_D - X_D^k\|^2 \Big\}, \\
g_k(X_P) &:= \inf_{X_P} \Big\{L(X_P;  X_D) - \frac{1}{2\sigma_k}\|X_D - X_D^k\|^2 \Big\} \\
& = \inprod{b}{y} + \inprod{z}{d} - \mu\inprod{I}{\textrm{exp}(-\frac{1}{\mu}(C-\mathcal{A}^*y-\mathcal{B}^*z-S+\mu I))} + \varepsilon\inprod{I}{S}\\
& - \frac{1}{2\sigma_k}(\|y-y^k\|^2 + \|z-z^k\|^2 + \|S-S^k\|^2).
\end{aligned}
\end{equation*}

To obtain $X_{P}^{k+1}$, based on the discussion in Subsection \ref{subsec:ALM}, we need the following criteria on the approximate computation of the subproblem \eqref{SSNAL-subproblem-entropy}:
\begin{eqnarray*}
	(A1) & f_{k} (X_P^{k+1}) - \inf f_{k}(X_P) \leq \frac{\varepsilon_k^2}{2 \sigma_k},\, \varepsilon_{k} \ge 0, \, \sum_{k=0}^{\infty} \varepsilon_k < \infty,  \label{stop1}  \\
	(B1) & f_{k} (X_P^{k+1}) - \inf f_{k}(X_P) \leq \frac{\delta_k^2}{2\sigma_k}\|X_D^{k+1} - X_D^k\|^2, 0 \leq \delta_k \leq 1,\nonumber\\
	&\sum_{k=0}^{\infty} \delta_k < \infty. \label{stop2}
\end{eqnarray*}
However, the stopping criteria (A1) and (B1) cannot be used directly since $\inf f_{k}(X_P)$ is unknown. Since
\begin{equation*}
\inf f_k (X_P) = \sup g_k(X_D) \ge g_k(X_D^{k+1}),
\end{equation*}
we have
\begin{equation*}
f_{k} (X_P^{k+1}) - \inf f_k (X_P) \leq f_{k} (X_P^{k+1}) - g_k(X_D^{k+1}).
\end{equation*}
Hence, we terminate the subproblem \eqref{SSNAL-subproblem-entropy} if $(X_P^{k+1}, X_D^{k+1})$ satisfies the following conditions:
\begin{eqnarray*}
	(A2)& f_{k} (X_P^{k+1}) - g_k(X_D^{k+1}) \leq \frac{\varepsilon_k^2}{2 \sigma_k}, \, \varepsilon_{k} \ge 0, \, \sum_{k=0}^{\infty} \varepsilon_k < \infty,  \label{stop3}   \\
	(B2)&  f_{k} (X_P^{k+1}) - g_k(X_D^{k+1}) \leq \frac{\delta_k^2}{2\sigma_k}\|X_D^{k+1} - X_D^{k}\|^2,  \, 0 \leq \delta_k \leq 1,\nonumber \\
	& \sum_{k=0}^{\infty} \delta_k < \infty. \label{stop4}
\end{eqnarray*}

As for solving the inner subproblem \eqref{SSNAL-subproblem-entropy}, we present the SSN method as below.
\begin{algorithm}
	\caption{SSN method for problem \eqref{SSNAL-subproblem}}\label{alg-ssn-entropy}
	Input $\sigma>0$, $y\in\mathcal{R}^{m_{1}}$, $z\in\mathcal{R}^{m_{2}}$
	$\nu\in(0,\frac{1}{2})$,  $\nu_{1},\nu_{2},\bar{\eta}\in(0,1)$,
	$\tau\in(0,1]$ and $\delta\in(0,1)$. Choose $X^{0},S^{0}\in\mathcal{S}^{n}$. Set $j=0$ and iterate:
	\begin{description}
		\item [Step 1] Let $W^{j}\in\partial\mathrm{\Pi}_{\mathcal{R}^{m_2}_{+}}(z-\sigma(\mathcal{B}(X^{j})-d))$, $\mathcal{V}^{j}\in\partial\mathrm{\Pi}_{\mathcal{S}^{n}_{+}}(S-\sigma(X^{j}-\varepsilon I))$, and $\mathcal{H}^{j}=(\mu\log)' X  + \sigma\mathcal{A}^*\mathcal{A} + \sigma\mathcal{B}^*W^{j}\mathcal{B} + \sigma\mathcal{V}^{j}$. Solve the following linear system
		\begin{eqnarray*}
			(\mathcal{H}^{j}+\epsilon_{j}\mathcal{I})(\Delta X)=-\nabla L_{\sigma}(X;\tilde{y},\tilde{z},\widetilde{S})
		\end{eqnarray*}
		by the preconditioned conjugate gradient (PCG) method to obtain an approximate solution $\Delta X^{j}$ satisfying the condition below
		\begin{eqnarray*}
			\|(\mathcal{H}^{j}+\epsilon_{j}\mathcal{I})(\Delta X^{j})+\nabla L_{\sigma}(X^{j};\tilde{y},\tilde{z},\widetilde{S})\|\leq
			\eta_{j}:=\min(\bar{\eta},\|\nabla L_{\sigma}(X^{j};\tilde{y},\tilde{z},\widetilde{S})\|^{1+\tau}),%\label{ineq:stopping criterion 1}
		\end{eqnarray*}
		where $\epsilon_{j}=\nu_{1}\min\{\nu_{2},\|\nabla L_{\sigma}(X^{j};\tilde{y},\tilde{z},\widetilde{S})\|\}$.
		\item [Step 2]  Set $\alpha_{j}=\delta^{m_{j}}$, where $m_{j}$ is the first nonnegative integer $m$ such that
		\begin{eqnarray*}
			&L_{\sigma}(X^{j}+\delta^{m}\Delta X^{j};\tilde{y},\tilde{z},\widetilde{S})\leq
			L_{\sigma}(X^{j};\tilde{y},\tilde{z},\widetilde{S})+\nu\delta^{m}\langle\nabla
			L_{\sigma}(X^{j}+\delta^{m}\Delta X^{j};\tilde{y},\tilde{z},\widetilde{S}),\Delta X^{j}\rangle.&
		\end{eqnarray*}
		\item [Step 3.]  Set $X^{j+1}=X^{j}+\alpha_{j}\Delta X^{j}$. If a desired stopping criterion is satisfied, terminate; otherwise set $j=j+1$ and go to Step 1.
	\end{description}
\end{algorithm}

\section{Numerical experiments}
\label{sec:Numerical experiments}

In this section, we implement the ALM on von Neumann entropy optimization problems. We implement the algorithms in {\sc Matlab} R2019a. All runs are performed on a NoteBook (i710710u 4.7G with 16 GB RAM).

For Algorithm \ref{alg-ssnal-entropy}, we adopt the following stopping criterion:
\begin{eqnarray*}
	R_{\mathrm{KKT}}:=\max\{R_{\mathrm{P}},R_{\mathrm{D}},R_{\mathrm{C}}\} < \mathrm{Tol}=10^{-6},
\end{eqnarray*}
where
\begin{eqnarray*}
	R_{\mathrm{P}}&:=&\max\Big\{\frac{\norm{\mathcal{A}(X)-b}}{1+\norm{b}},\frac{\norm{\mathcal{B}(X)-d-\mathrm{\Pi}_{\mathcal{R}^{m_2}_{+}}(\mathcal{B}(X)-d)}}{1+\norm{\mathcal{B}(X)-d}},\\
	&&\frac{\norm{X-\varepsilon I-\mathrm{\Pi}_{\mathcal{S}^{n}_{+}}(X-\varepsilon I)}}{1+\norm{X-\varepsilon I}}\Big\},\\
	R_{\mathrm{D}}&:=&\frac{\norm{C-\mathcal{A}^*y-\mathcal{B}^*z-S+\mu\log X + \mu I}}{1+\norm{C}},\\
	R_{\mathrm{C}}&:=&\max\Big\{\frac{\norm{\mathcal{B}(X)-d-\mathrm{\Pi}_{\mathcal{R}^{m_2}_{+}}(\mathcal{B}(X)-d-z)}}{1+\norm{z}},\frac{\norm{X-\varepsilon I-\mathrm{\Pi}_{\mathcal{S}^{n}_{+}}(X-\varepsilon I-S)}}{1+\norm{S}}\Big\}.
\end{eqnarray*}

\subsection{PALM}
\label{subsec:PALM}
In practical computation, it is better to employ the PALM to warmstart the ALM. Therefore, we usually implement the PALM to iterate $200$ steps to generate an initial point. Now we present the details of the PALM as follows:
\begin{algorithm}
	\caption{PALM for problem \eqref{eq:entropy-primal-problem}}\label{alg-PALM-entropy}
	Given $\sigma>0$, $\tau\in(0,2)$ and $X^{0}\in\mathcal{S}^{n}$, $y^{0}\in\mathcal{R}^{m_{1}}$, $u^{0},z^{0}\in\mathcal{R}^{m_{2}}$. For $k=0,1,2,\ldots$, iterate the following steps:
	\begin{description}
		\item[1] Compute
		\begin{eqnarray*}
			\alpha_{k} = \lambda_{\max}(\sigma\mathcal{A}^*\mathcal{A} + \sigma\mathcal{B}^*\mathcal{B}).
		\end{eqnarray*}
		\item[2] Set $\mathcal{T}_{k}=\alpha_{k}\mathcal{I} - \sigma\mathcal{A}^*\mathcal{A} - \sigma\mathcal{B}^*\mathcal{B}$ and $M^{k}=\mathcal{A}^*y^{k}+\sigma\mathcal{A}^*b+\mathcal{B}^*z^{k}+\sigma\mathcal{B}^*(d+u^{k})+\mathcal{T}(X^{k})-C$. Compute
		\begin{eqnarray*}
			X^{k+1}&=&\arg\min_{Y}\Big\{\frac{1}{2}\norm{X-M^{k}}^{2}+\frac{\mu}{\sigma}\textrm{Tr}(X\log X)\,|\,X\succeq\varepsilon I\Big\},\\
			u^{k+1}&=&\mathrm{\Pi}_{\mathcal{R}^{m_2}_{+}}(\mathcal{B}(X^{k+1})-d-\frac{1}{\sigma}z^{k})).
		\end{eqnarray*}
		\item[3] Compute
		\begin{eqnarray*}
			y^{k+1}&=&y^{k}-\tau\sigma(\mathcal{A}(X^{k+1})-b),\\
			z^{k+1}&=&z^{k}-\tau\sigma(\mathcal{B}(X^{k+1})-u^{k+1}-d).
		\end{eqnarray*}
	\end{description}
\end{algorithm}

\subsection{The random data problem}
\label{subsec:random-data-problem}
In this subsection, we consider the following problem
\begin{eqnarray*}
	(P_{0}) & & \min_{X} \Big\{
	\inprod{C}{X}+\mu\trace(X\log X) \,\mid
	\,\trace(X)=1,\, X_{ij} = 0,\,i=1,\cdots,m, \;\;
	X \succeq \varepsilon I \Big \},
	\label{P0}
\end{eqnarray*}
where $\mathcal{E}$ is an index set, and it is generated randomly.
The {\sc Matlab} code to generate the data matrix $C$ is
\begin{eqnarray*}
	& & {\tt  x = 10.\textrm{\^{}} [-4:4/(n-1):0];\  C = gallery('randcorr',n*x/sum(x));} \\[5pt]
	& & {\tt  E = 2.0*rand(n,n) - ones(n,n);\  E = triu(E) + triu(E,1)';} \\[5pt]
	& & {\tt  for\ i=1:n;\ E(i,i) =1;\ end} \\[5pt]
	& & {\tt  alpha = .1;\  C = (1-alpha)*C + alpha*E;} \\[5pt]
	& & {\tt  C = (C+C')/2;\  C = min(C,ones(n,n));\ C = max(-ones(n,n),C);} \\[5pt]
\end{eqnarray*}
We compare the the numerical performances between the ALM and the PALM on the random data problems. The performances of the algorithms are
presented in Tables \ref{tablerand-1}-\ref{tablerand-2}. For each instance in the tables, we report the number of outer iterations
({\em it}), the total number of subproblems ({\em itsub}), and the
average number of PCG steps ({\em pcg}) taken to solve each linear
system (for the ALM); the number of the iterations
({\em it}) (for the PALM); the primal (\pobj) and dual (\dobj) objective
values; the relative primal infeasibility ($R_P$), the relative dual infeasibility ($R_D$), the relative complementarity condition ($R_C$), and the relative gap ($R_G$); the computing time (time) in the format of ``hours:minutes:seconds''. For simplicity, we use``$s\ \mbox{sign}(t)|t|$'' to denote a number of the form ``$s\times 10^{t}$'', e.g., 1.0-3 denotes $1.0\times 10^{-3}$. We present the results of every problem in two rows: the result for the PALM in the first row; and the result for the ALM in the second row.

In Table \ref{tablerand-1}, it is evident that the ALM successfully solves all the problems with the required accuracy within a reasonable amount of time. In contrast, the PALM fails to achieve the desired accuracy for any of the problems. Although the ALM completes the tasks, it requires a significant number of Newton iterations and PCG iterations. This is primarily due to the ill-conditioned nature of the inner subproblem, which arises from the presence of the entropy term. We also observe that in general the problem becomes increasingly ill-conditioned as $\mu$ decreases.

In Table \ref{tablerand-2}, we observe a general trend where the problem becomes better conditioned as the number of equality constraints increases. This is because, as the number of constraints grows, the term $\sigma\mathcal{A}^*\mathcal{A}$ becomes denser, which improves the nature of the inner subproblem. This indicates that our proposed algorithm is highly effective for solving problems with a large number of constraints.

\subsection{Matrix nearness problems with Bregman divergences}
\label{subsec:Bregman divergences}

In this subsection, we consider the matrix nearness problem
\begin{eqnarray*}
	(\textrm{BR}) & & \min_{X} \Big\{
	D_{\phi}(X,X_{0}) \,\mid
	\, \mathcal{A}(X) = b, \;\;
	X \succeq \varepsilon I \Big \},
	\label{Bregdiv problem}
\end{eqnarray*}
where
\begin{eqnarray*}
	D_{\phi}(X,X_{0}) &=& \phi(X)-\phi(X_{0})-\inprod{\nabla\phi(X_{0})}{X-X_{0}}
	\label{Bregdiv}
\end{eqnarray*}
is the so-called Bregman divergence. If $\phi$ takes different
functions, we can obtain different Bregman divergences. If
$\phi(X)=\frac{1}{2}\|X\|^{2}_{F}$, the associated divergence is the
squared Frobenius norm
$D_{\phi}(X,X_{0})=\frac{1}{2}\|X-X_{0}\|^{2}_{F}$; if $\phi(X)$
equals the von Neumann entropy $\trace(X\log X)$, we obtain the von
Neumann divergence
\begin{eqnarray*}
	D_{vN}(X,X_{0}) &=& \trace(X\log X-X\log X_{0}-X+X_{0}).
	\label{von_Neumann_div}
\end{eqnarray*}
Bregman divergences are well suited for metric nearness problems
because they share many geometric properties with the squared
Frobenius norm. They also exhibit an intimate relationship with
exponential families of probability distributions, which recommends
them for solving problems that arise in the statistical analysis of
data. For more details of (BR), one may refer to \cite{KulisSD}.

The nearest correlation matrix problem \cite{Higham} is a typical
case of (BR) that arises in various fields. A correlation matrix is
a positive semidefinite matrix with unit diagonal. However, in
\cite{Higham} the Euclidean distance is used which may lead to the
result that the structure of the computational optimal solution is
destroyed. For example, the rank of the optimal solution may be far
from that of $X_{0}$. But if a rank constraint $\rank(X) \leq r$, where $r>0$, is added to the problem,
then it becomes a nonconvex one.

The von Neumann divergence owns some nice properties. Among
which the rank-keeping property is a very interesting one, that is,
the rank constraint $\rank(X) \leq r$ is automatically satisfied when the rank of $X_{0}$ does not exceed
$r$, then this problem turns out to be convex. This is because this kind of divergences restrict the search for the optimal $X$ to the linear subspace of matrices that have the same range space as $X_0$.

Now we adopt the von Neumann divergence instead to find the nearest correlation matrices.
Then, the problem (BR) can be regarded as a special case of the problem ($P$), so the proposed algorithm can be applied to it.

Below is the testing example.
The matrix $X_{0}$ is an estimated $943\times943$ correlation matrix based on $100, 000$ ratings
for $1682$ movies by $943$ users. Due to the missing data, the
generated matrix $X_{0}$ is not positive semidefinite
\cite{Fushiki}, we have to project it onto the positive semidefinite
cone to adapt to the von Neumann divergence. This rating data set
can be downloaded from\\ \texttt{http://www.grouplens.org/node/73}.
Besides the unit diagonal constraints,
we optionally add sparse constraints to the problem.\\

The performances of the two algorithms on the testing example
are listed in Table \ref{table-movielens}. From the computing results we can see that the ALM can compute
the problems in a reasonable amount of time. Moreover, the rank of the computed solution $X$ is less than that of the matrix $X_0$. In contrast, the PALM fails to solve the problem. Moreover, the solution also fails to satisfy the rank requirement.

\subsection{Kernel learning problems}
\label{subsec:kernel learning }

The diffusion kernel is a general method for computing pairwise
distances among all nodes in a graph, based on the sum of weighted
paths between each pair of nodes. This technique has been used
successfully, in conjunction with kernel-based learning methods, to
draw inferences from several types of biological networks. Given $n$
nodes in a graph $x_{1},\cdots,x_{n}$,  the positive semidefinite
kernel matrix $K$ can be derived as the optimal solution of the
maximum entropy problem subject to the distance constraints
\begin{eqnarray}
& & \|x_{s_{j}}-x_{t_{j}}\|^{2}\leq \gamma,\ j=1,\cdots,m_2,
\label{distance constraint}
\end{eqnarray}
where $\{s_{j},t_{j}\}_{j=1}^{m_2}$ denote the node pairs connected by
$m_2$ edges. Once a kernel matrix is determined, the (squared)
Euclidean distance between two points can also be computed as
$\|x_{i}-x_{j}\|^{2}=K_{ii}+K_{jj}-2K_{ij}$. The problem can be
formulated as
\begin{eqnarray*}
	(\textrm{KL})& & \min_{K} \Big\{
	\trace(K\log K) \,\mid
	\, \trace(K)=1, \, \trace(KV_{j})\leq\gamma, j=1,\cdots,m_2, \;\;
	K \succeq 0 \Big \},
	\label{kernel learning}
\end{eqnarray*}
where
\begin{eqnarray*}
	[V_{j}]_{st} = & & \left\{ \begin{array}{ll}
		1,  &\mbox{if $(s=s_{j},\, t=s_{j})$ or $(s=t_{j},\, t=t_{j})$} \\[5pt]
		-1, &\mbox{if $(s=s_{j},\, t=t_{j})$ or $(s=t_{j},\, t=s_{j})$} \\[5pt]
		0,  &\mbox{otherwise.}
	\end{array} \right.
	\label{V}
\end{eqnarray*}
For the details of the kernel learning problem, one may see
\cite{TsudaN}.

Now we compute kernels from a yeast biological
network. This network was created by von Mering et al.
\cite{vMeringKSCOFB} from protein-protein interactions identified
via six different methods: high-throughput yeast two-hybrid,
correlated mRNA expression, genetic interaction (synthetic
lethality), tandem affinity purification, high-throughput
massspectrometric protein complex identification and computational
methods. All interactions were classified into one of three
confidence categories, high-, medium- and lowconfidence, based on
the number of different methods that identify an interaction as well
as the number of times the interaction is observed. In these
experiments, we use a medium confidence network containing $1000$
proteins and $1915$ edges.

The computing results can be seen in Table \ref{table-KL}. The ALM can successfully solve the problem, however, the PALM cannot supply a satisfying solution within the maximal iterations.

\section{Conclusion}
\label{sec:Conclusion}
In this paper, we focus on a class of convex composite optimization problems. We prove the equivalence between the primal/dual SSOSC and the dual/primal nondegeneracy condition, and then establish a specific set of equivalent conditions for the perturbation analysis of the problem. Furthermore, we establish the equivalence between the primal/dual SOSC and the dual/primal SRCQ, as well as the equivalence between the dual nondegeneracy condition and the nonsingularity of Clarke's generalized Jacobian of the subproblem of the ALM. These results provide a solid theoretical foundation for the ALM. Finally, we take the von Neumann entropy optimization problem as an example to demonstrate the effectiveness of the ALM.

\begin{landscape}
	\begin{center}
		\begin{footnotesize}
			\begin{longtable}{| c | c | c | ccc | cc|}
				\caption{Performances of the two algorithms on the random data problems ($P_{0}$) with different sizes}
				\label{tablerand-1}
				\\ \hline
				\mc{1}{|c|}{} & \mc{1}{|c|}{} & \mc{1}{|c|}{} & \mc{3}{c|}{}  & \mc{2}{c|}{} \\[-8pt]
				\mc{1}{|c|}{algorithm} & \mc{1}{|c|}{pbname $(n,m_1)$} &\mc{1}{|c}{$\mu$} &\mc{1}{c}{{\em it}/{\em
						itsub}/{\em pcg}} &\mc{1}{c}{\pobj} &\mc{1}{c|}{\dobj}
				&\mc{1}{c}{$R_P$/$R_D$/$R_C$/$R_G$} &\mc{1}{c|}{time}
				\\[2pt] \hline
				\endhead
				
				\hline
				\endfoot

				PALM & rand(500,121976) & 1 & 10000 & -1.9412551-13 & -2.0799913+2
				& 1.0+0 $|$ 1.3+2 $|$ 2.0-6 $|$ 9.9-1 & 00:17:03 \\[2pt]
				ALM & rand(500,121976)   & 1 & 24 $|$ 229 $|$ 3.8 & -5.2628536+0 & -5.2628527+0
				& 8.8-7 $|$ 8.7-9 $|$ 0.0+0 $|$ 7.2-8 & 00:01:47 \\[2pt]
				\hline
				
				PALM & rand(500,121976)  & 0.1 & 10000 & -8.1760343-4 & -2.0972557+2
				& 9.9-1 $|$ 2.4+2 $|$ 4.2-21 $|$ 9.9-1 & 00:17:34 \\[2pt]
				ALM & rand(500,121976)   & 0.1 & 27 $|$ 295 $|$ 29.1 & 1.9428636-1 & 1.9428850-1
				& 5.0-7 $|$ 9.7-7 $|$ 1.9-10 $|$ 8.5-7 & 00:10:58 \\[2pt]
				\hline
				
				PALM & rand(1000,499396) & 1 & 10000 & -1.3650948-13 & -3.7711096+2
				& 1.0+0 $|$ 2.0+2 $|$ 1.4-6 $|$ 9.9-1 & 03:21:22 \\[2pt]
				ALM & rand(1000,499396)   & 1 & 22 $|$ 269 $|$ 5.6 & -5.9084062+0 & -5.9084062+0
				& 7.9-7 $|$ 2.3-8 $|$ 0.0+0 $|$ 8.4-10 & 00:14:40 \\[2pt]
				\hline
				
				PALM & rand(1000,499396)  & 0.1 & 10000 & -1.6393827-3 & -3.7840614+2
				& 9.8-1 $|$ 5.0+2 $|$ 2.0-21 $|$ 9.9-1 & 01:55:01 \\[2pt]
				ALM & rand(1000,499396)   & 0.1 & 29 $|$ 329 $|$ 18.7 & 3.0305337-1 & 3.0305330-1
				& 4.6-7 $|$ 5.5-8 $|$ 2.1-8 $|$ 2.5-8 & 00:43:32 \\[2pt]
				\hline
				
				PALM & rand(1500,1124241)  & 1 & 10000 & -2.9202916-1 & -4.7626308+2
				& 9.7-1 $|$ 2.3+2 $|$ 2.8-8 $|$ 9.9-1 & 07:17:23 \\[2pt]
				ALM & rand(1500,1124241)   & 1 & 26 $|$ 310 $|$ 6.6 & -6.3132698+0 & -6.3132698+0
				& 3.6-8 $|$ 2.9-9 $|$ 0.0+0 $|$ 3.8-10 & 01:11:16 \\[2pt]
				\hline
				
				PALM & rand(1500,1124241)  & 0.1 & 10000 & -2.4575716-3 & -4.7647815+2
				& 9.7-1 $|$ 6.5+2 $|$ 2.8-8 $|$ 1.0+0 & 05:00:03 \\[2pt]
				ALM & rand(1500,1124241)   & 0.1 & 21 $|$ 284 $|$ 3.5 & 2.6813361-1 & 2.6813362-1
				& 9.3-7 $|$ 2.4-7 $|$ 0.0+0 $|$ 3.2-9 & 00:44:21 \\[2pt]
				\hline
				
				PALM & rand(2000,1996226) & 1 & 10000 & -3.9020421-1 & -7.2251358+2
				& 9.6-1 $|$ 3.3+2 $|$ 2.1-8 $|$ 1.0+0 & 21:28:27 \\[2pt]
				ALM & rand(2000,1996226)   & 1 & 26 $|$ 326 $|$ 5.0 & -6.6074188+0 & -6.6074188+0
				& 5.3-7 $|$ 3.6-9 $|$ 0.0+0 $|$ 2.8-11 & 02:06:38 \\[2pt]
				\hline
				
				PALM & rand(2000,1996226) & 0.1 & 10000 & -3.2784984-3 & -7.2538523+2
				& 9.6-1 $|$ 1.0+3 $|$ 2.5e-21 $|$ 1.0+0 & 19:25:02 \\[2pt]
				ALM & rand(2000,1996226)   & 0.1 & 39 $|$ 403 $|$ 35.9 & 2.0087568-1 & 2.0087535-1
				& 7.1-7 $|$ 2.3-7 $|$ 6.8-10 $|$ 1.4-7 & 17:29:15 \\[2pt]
				\hline
			\end{longtable}
		\end{footnotesize}
	\end{center}
\end{landscape}

\begin{landscape}
	\begin{center}
		\begin{footnotesize}
			\begin{longtable}[!ht]{| c | c | c | ccc | cc|}
				\caption{Performances of the two algorithms on the random data problems ($P_{0}$) with $n=1000$ and different $m_1$}
				\label{tablerand-2}
				\\ \hline
				\mc{1}{|c|}{} & \mc{1}{|c|}{} & \mc{1}{|c|}{} & \mc{3}{c|}{}  & \mc{2}{c|}{} \\[-8pt]
				\mc{1}{|c|}{algorithm} & \mc{1}{|c|}{pbname $(n,m_1)$} &\mc{1}{|c}{$\mu$} &\mc{1}{c}{{\em it}/{\em
						itsub}/{\em pcg}} &\mc{1}{c}{\pobj} &\mc{1}{c|}{\dobj}
				&\mc{1}{c}{$R_P$/$R_D$/$R_C$/$R_G$} &\mc{1}{c|}{time}
				\\[2pt] \hline
				\endhead
				
				\hline
				\endfoot

				PALM & rand(1000,94951) & 1 & 10000 & -2.8100387-13 & -2.5952424+2
				& 1.0+0 $|$ 9.7+1 $|$ 2.9-6 $|$ 9.6-1 & 01:32:33 \\[2pt]
				ALM & rand(1000,94951)   & 1 & 27 $|$ 349 $|$ 25.0 & -7.3551273+0 & -7.3551271+0
				& 9.7-8 $|$ 1.6-8 $|$ 7.7-9 $|$ 7.3-9 & 01:00:41 \\[2pt]
				\hline
				
				PALM & rand(1000,134416)  & 1 & 10000 & -2.4273168-13 & -2.8569898+2
				& 1.0+0 $|$ 1.1+2 $|$ 2.2-6 $|$ 9.7-1 & 01:32:42 \\[2pt]
				ALM & rand(1000,134416)   & 1 & 28 $|$ 362 $|$ 26.6 & -7.2296844+0 & -7.2296844+0
				& 4.5-8 $|$ 2.7-7 $|$ 6.1-12 $|$ 7.8-10 & 01:27:34 \\[2pt]
				\hline
				
				PALM & rand(1000,199576) & 1 & 10000 & -1.5068785-1 & -3.1231482+2
				& 8.9-1 $|$ 1.2+2 $|$ 1.6-6 $|$ 9.8-1 & 01:38:11 \\[2pt]
				ALM & rand(1000,199576)   & 1 & 28 $|$ 354 $|$ 22.7 & -7.0329644+0 & -7.0329641+0
				& 6.6-8 $|$ 5.5-9 $|$ 7.5-12 $|$ 1.6-8 & 01:22:53 \\[2pt]
				\hline
				
				PALM & rand(1000,386926)  & 1 & 10000 & -8.2704422-2 & -4.5938212+2
				& 9.9-1 $|$ 2.1+2 $|$ 9.2-7 $|$ 9.9-1 & 01:25:12 \\[2pt]
				ALM & rand(1000,386926)   & 1 & 25 $|$ 326 $|$ 21.9 & -6.4191964+0 & -6.4191933+0
				& 6.3-7 $|$ 3.1-7 $|$ 1.5-10 $|$ 1.9-7 & 00:59:13 \\[2pt]
				\hline
				
				PALM & rand(1000,461826)  & 1 & 10000 & -1.4149891-13 & -4.2611580+2
				& 1.0+0 $|$ 2.2+2 $|$ 1.3-6 $|$ 9.9-1 & 02:21:57 \\[2pt]
				ALM & rand(1000,461826)   & 1 & 27 $|$ 315 $|$ 9.7 & -6.1094205+0 & -6.1094184+0
				& 6.2-7 $|$ 1.2-8 $|$ 6.0-10 $|$ 1.4-7 & 00:23:59 \\[2pt]
				\hline
				
				PALM & rand(1000,499396)  & 1 & 10000 & -1.3650948-13 & -3.7711096+2
				& 9.8-1 $|$ 5.0+2 $|$ 2.0-21 $|$ 9.9-1 & 01:55:01 \\[2pt]
				ALM & rand(1000,499396)   & 1 & 29 $|$ 329 $|$ 18.7 & 3.0305337-1 & 3.0305330-1
				& 4.6-7 $|$ 5.5-8 $|$ 2.1-8 $|$ 2.5-8 & 00:43:32 \\[2pt]
				\hline
				
			\end{longtable}
		\end{footnotesize}
	\end{center}
\end{landscape}
\begin{landscape}
	\begin{center}
		\begin{footnotesize}
			\begin{longtable}{| c | c | cccc | cc|}
				\caption{Performances of the two algorithms on the movielens data problem (BR) with $n=943$ and $m_1=443776$.}
				\label{table-movielens}
				\\ \hline
				\mc{1}{|c|}{} & \mc{1}{|c|}{} & \mc{4}{c|}{}  & \mc{2}{c|}{} \\[-8pt]
				\mc{1}{|c|}{algorithm} & \mc{1}{|c|}{$(n,m_1)$} &\mc{1}{c}{{\em it}/{\em
						itsub}/{\em pcg}} &\mc{1}{c}{\pobj} &\mc{1}{c}{\dobj} &\mc{1}{c|}{rank$X_{0}$$|$rank$X$}
				&\mc{1}{c}{$R_P$/$R_D$/$R_C$/$R_G$} &\mc{1}{c|}{time}
				\\[2pt] \hline
				\endhead
				
				\hline
				\endfoot

				PALM & (943,443776) & 10000 & 1.2370463-1 & -9.4347874+3 & 453$|$943
				& 9.8-1 $|$ 3.8+2 $|$ 1.4-22 $|$ 9.9-1 & 03:10:21 \\[2pt]
				ALM & (943,443776)  & 62 $|$ 490 $|$ 30.5 & 1.0030991+1 & 1.0030999+1 & 453$|$20
				& 2.8-7 $|$ 7.7-7 $|$ 3.9-11 $|$ 1.1-7 & 02:37:31 \\[2pt]
				\hline
				
			\end{longtable}
		\end{footnotesize}
	\end{center}
	%\end{landscape}
	%
	%\begin{landscape}
	\begin{center}
		\begin{footnotesize}
			\begin{longtable}{| c | c | ccc | cc|}
				\caption{Performances of the two algorithms on the protein-protein interactions data problem (KL) with $n=1000$, $m_1=1$ and $m_2=1915$.}
				\label{table-KL}
				\\ \hline
				\mc{1}{|c|}{} & \mc{1}{|c|}{} & \mc{3}{c|}{}  & \mc{2}{c|}{} \\[-8pt]
				\mc{1}{|c|}{algorithm} & \mc{1}{|c|}{$(n,m_1,m_2)$} &\mc{1}{c}{{\em it}/{\em
						itsub}/{\em pcg}} &\mc{1}{c}{\pobj} &\mc{1}{c|}{\dobj}
				&\mc{1}{c}{$R_P$/$R_D$/$R_C$/$R_G$} &\mc{1}{c|}{time}
				\\[2pt] \hline
				\endhead
				
				\hline
				\endfoot

				PALM & (1000,1,1915) & 10000 & -6.0224289-14 & -4.4677547+1
				& 1.0+0 $|$ 8.7+0 $|$ 1.9-3 $|$ 1.0+0 & 01:10:11 \\[2pt]
				ALM & (1000,1,1915)  & 65 $|$ 492 $|$ 28.0 & -6.3537544+0 & -6.3537539+0
				& 8.8-7 $|$ 1.0-10 $|$ 4.9-8 $|$ 4.0-8 & 02:47:16 \\[2pt]
				\hline
				
			\end{longtable}
		\end{footnotesize}
	\end{center}
\end{landscape}
\afterpage{\clearpage}

\bibliographystyle{splncs03}
\bibliography{ref_1} 
% common bib file
%% if required, the content of .bbl file can be included here once bbl is generated
%%\input sn-article.bbl

\end{document}